\documentclass[12pt,reqno]{amsart}
\usepackage{amsfonts}
\usepackage{a4wide}
\allowdisplaybreaks
\numberwithin{equation}{section}

\newtheorem{theorem}{Theorem}[section]
\newtheorem{proposition}[theorem]{Proposition}

\newtheorem{lemma}[theorem]{Lemma}

\theoremstyle{definition}

\newtheorem{remark}[theorem]{Remark}

\newcommand{\R}{\mathbb{R}}
\newcommand{\dis}{\displaystyle}
\newcommand{\va}{\varepsilon}

\begin{document}

\title
[Segregated  Vector Solutions]{Segregated   Vector Solutions for
linearly coupled  Nonlinear Schr\"odinger Systems}

\author{Chang-Shou Lin \, and \,Shuangjie Peng
}

\address{Taida Institute for Mathematical Sciences and Department of Mathematics, National Taiwan University, Taipei, 10617, Taiwan}

\email{cslin@math.ntu.edu.tw }

\address{School of Mathematics and Statistics, Central China
Normal University, Wuhan, 430079, P. R. China }

\email{ sjpeng@mail.ccnu.edu.cn}

\begin{abstract}
We  consider the following  system linearly coupled by nonlinear
Schr\"odinger equations in $\R^3$
$$
\left\{
\begin{array}{ll}
-\Delta u_j+u_j=u^3_j-\va\sum\limits_{i\neq j}^N u_i,\hspace{1cm}& x\in \R^3, \vspace{0.2cm}\\
u_j\in H^1(\R^3),\quad j=1,\cdots,N,
\end{array}
\right.
$$
where $\va\in\R$ is a
 coupling constant. This type of system arises in particular in models in nonlinear $N$-core fiber.

We examine the effect of the linear coupling to the solution
structure. When $N=2,3$, for any prescribed integer $\ell\ge 2$, we
construct a non-radial  vector solutions of segregated type, with
two components having exactly $\ell$ positive bumps for $\va>0$
sufficiently small. We also give an explicit description on the
characteristic features of the vector solutions.

\end{abstract}

\maketitle

\section{Introduction}
We consider the following nonlinear Schr\"odinger systems which are
linearly coupled by $N$ equations
\begin{equation}\label{eq}
\left\{
\begin{array}{ll}
-\Delta u_j+u_j=u^3_j-\va\sum\limits_{i\neq j}^N u_i,\hspace{1cm}& x\in \R^3, \vspace{0.2cm}\\
u_j\in H^1(\R^3),\quad j=1,\cdots,N,
\end{array}
\right.
\end{equation}
where $\va\in\R$.  These systems arise when one considers stationary
pulselike (standing wave) solutions of the time-dependent
$N$-coupled Schr\"odinger systems of the form
\begin{equation}\label{E1}\left\{\begin{array}{l}
-i\frac{\partial}{\partial t}\Phi_j=\Delta\Phi_j-\Phi_j+
|\Phi_j|^2\Phi_j-\va\sum^N\limits_{i\neq j}\Phi_i,
\ \ \hbox{in}\ \R^3\times\R^+,\\
\Phi_j=\Phi_j(x,t)\in \Bbb C,  t>0, \ j=1,\cdots, N.\\
\end{array}\right.\end{equation}
 This type of system arises in nonlinear optics. For example, the
propagation of optical pulses in nonlinear $N$-core directional
coupler  can be described by $N$ linearly coupled nonlinear
Schr\"odinger equations. Here $\Phi_j$ ($j=1,\cdots,N$) are envelope
functions and  $\va$, which is the normalized coupling coefficient
between the cores, is equal to the linear coupling coefficient times
the dispersion length. The sign of $\va$ determines whether the
interactions of fiber couplers are repulsive or attractive. In the
attractive case the components of a vector solution tend to go along
with each other leading to synchronization, and in the repulsive
case the components tend to segregate with each other leading to
phase separations. These phenomena have been documented in numeric
simulations (e.g., \cite{AA} and references therein).

Nonlinear Schr\"{o}dinger equations have been broadly investigated
in many aspects, such as existence of solitary waves, concentration
and multi-bump phenomena for semiclassical states (see e.g.
\cite{Cao}, \cite{WY} and the references therein). The study on
system of Schr\"odinger equations began quite recently. Mathematical
work on systems with the nonlinearly coupling terms (e.g. the term
$\sum_{i\neq j}^N u_i$ in \eqref{eq} being replaced by
$u_j\sum_{i\neq j}^N u_i^2$ ) has been studied extensively in
recent years, for example,
\cite{BDW,DWW,DL,LtW1,LtW2,LzW1,MPS,nttv,S,TV,WW1,WW2} and
references therein, where phase separation or synchronization has
been proved in several cases.

However, for the linearly coupled system \eqref{eq}, as far as the
authors know, it seems that there are few results. In  \cite{AR},
when $N=2$, solitons of linearly coupled systems of semilinear
non-autonomous equations were studied by using concentration
compactness principle, and existence of both positive ground and
bound states was proved
 under some decay  assumptions on the potentials at
infinity. In \cite{A}, this type of non-autonomous systems was also
considered by using a perturbation argument. Concerning on
autonomous systems, we also mention some results. If $N=2$ and the
dimension is one, for $\va<0$, \eqref{eq} has in addition to the
semi-trivial solutions $(\pm U,\,0),\,(0,\,\pm U)$, two types of
soliton like solutions, given by
\begin{eqnarray*}
&&(U_{1+\va},\,U_{1+\va}),\,\,(-U_{1+\va},\,-U_{1+\va}),\quad
\hbox{for}\,\,\,-1\le\va\le 0,\,\,(\hbox{symmetric}\,\,\,
\hbox{states}),\\
&&(U_{1-\va},\,-U_{1-\va}),\,\,(-U_{1-\va},\,U_{1-\va}),\quad
\hbox{for}\,\,\,\va\le 0,\,\,(\hbox{anti-symmetric}\,\,\,
\hbox{states}),
\end{eqnarray*}
where, for $\lambda>0$, $U_\lambda$ is the unique solution of
$$
\begin{cases}
-u'' +\lambda u= u^{3},\quad u>0, \quad \text{in}\; \R,\\
u(0)=\max\limits_{x\in\R} u(x),\,\,u(x)\in H^1(\R).
\end{cases}
$$
By using numerical methods, a bifurcation diagram is reported in
\cite{AA} where it is indicated that for $\va\in (-1,0)$, there
exists a family of new solutions for \eqref{eq}, bifurcating from
the branch of the anti-symmetric state at $\va=-1$. This kind of
results was rigorously verified in \cite{amco1} for small value of
the parameter $\va<0$. More precisely, in \cite{amco1}, it was
proved that a solution with one $2$-bump component having bumps
located near $\pm|\ln(-\va)|$, while the other
component having one negative peaks exists. This type of results was
generalized recently  in an interesting paper \cite{ACR} to two and
three dimensional cases. In \cite{ACR}, it was proved that if
$\mathcal{P}$ denotes a regular polytope centered at the origin of
$\R^d\,(d=2,3)$ such that its side is  larger than the radius of the
circumscribed circle or sphere, then there exists a solution with
one multi-bump component having bumps located near the vertices of
$\ln(-\va)\mathcal{P}$, while the other component has one negative
peak as $\va\to 0^-$. So in \cite{ACR}, the first component of the
solutions has more than one bump, while the second component is
negative and has only one bump. We emphasize here that the solutions
obtained in \cite{ACR} bifurcate also from the branch of
anti-symmetric state at $\va=0$. Furthermore, as  pointed out in
\cite{amco1}, for $\va<0$,  vector solutions with one component
being multi-bump do not exist near symmetric states, but only near
the anti-symmetric ones. Hence, an interesting problem is: can we
find solutions bifurcating from the symmetric state if $\va>0$? In
this paper, our main purpose is to prove that, for any prescribed
integer $\ell\ge 2$, \eqref{eq} has new solutions, different from
the previous ones, with the feature that two components have exactly
$\ell$ positive bumps when $\va>0$ is sufficiently small.

To state our  main results, we introduce some notations.

The Sobolev space $H^1(\R^3)$ is endowed with the standard norm
$$
\|u\|_{\R^3}=\Bigl(\int_{\R^3}(|\nabla
u|^2+u^2)\Bigl)^{\frac{1}{2}}.
$$

Denote by $U$ the unique solution of the following problem
\begin{equation}\label{10-18-4}
\begin{cases}
-\Delta u +u= u^{3},\quad u>0, \quad \text{in}\; \R^3,\\
u(0)=\max\limits_{x\in\R^3} u(x),\,\,u(x)\in H^1(\R^3).
\end{cases}
\end{equation}
It is well known that $U(x)=U(|x|)$ satisfies
$$
\lim\limits_{|x|\to+\infty}|x|e^{|x|}U=A>0,\,\,\hbox{and}\,\,\lim\limits_{|x|\to+\infty}\frac{U'(|x|)}{U(x)}=-1.
$$
Moreover, $U(x)$ is non-degenerate, that is,
$$
Kernel(\mathbb{L})=span\Bigl\{\frac{\partial U(x)}{\partial
x_i}:\,\,i=1,2,3\Bigl\},
$$
where $\mathbb{L}$ is the linearized operator
$$
\mathbb{L}:\,\,H^1(\R^3)\to L^2(\R^3),\,\,\mathbb{L}(u)=:\Delta
u-u+3U^2u.
$$

Let
\begin{equation}\label{1-29-9}
x^j=\dis\Bigl(r \cos\frac{2(j-1)\pi}\ell,
r\sin\frac{2(j-1)\pi}\ell,0\Bigr):=\bigl({x'}^j,\,\,0\bigr),\,\,
j=1,\cdots,\ell,
\end{equation}
and
\begin{equation}\label{1-29-10}
\begin{array}{l}
y^j=\dis\Bigl(\rho \cos\frac{(2j-1)\pi}\ell,
\rho\sin\frac{(2j-1)\pi}\ell,0\Bigr):=\bigl({y'}^j,\,\,0\bigr),\,\,\,
j=1,\cdots,\ell,
\end{array}
\end{equation}
 where   $r,\,\rho\in [r_0 |\ln\va|,\,\, r_1|\ln\va|]$ for some  $r_1>r_0>0$.

In this paper, for any function $W:\R^3\to\R$  and $\xi\in\R^3$, we
define $W_{\xi}=W(x-\xi)$.

We first consider  the following problem linearly coupled by two
nonlinear Schr\"{o}dinger equations
\begin{equation}\label{eqmain}
\left\{
\begin{array}{ll}
-\Delta u+u=u^3-\varepsilon v,\hspace{1cm}& x\in \R^3, \vspace{0.2cm}\\
-\Delta v+v=v^3 -\varepsilon u,& x\in\R^3.
\end{array}
\right.
\end{equation}

The main result can be stated as follows

\medskip
\begin{theorem} \label{main}
For any integer $\ell\ge 2$, there exists $\varepsilon_0$ such that
for $\va\in (0,\va_0)$, problem~\eqref{eqmain} has a solution
$(u,v)\in H^1(\R^3)\times H^1(\R^3)$ satisfying
$$
u^\va\sim \sum\limits_{j=1}^\ell U_{x_\va^j},\,\,v^\va\sim
\sum\limits_{j=1}^\ell U_{y_\va^j},
$$
 where $x_{\va}^j$ and $y_{\va}^j$ are respectively  defined by \eqref{1-29-9} and
 \eqref{1-29-10} with
 $$
r_\va =
\frac{2\sin\frac{\pi}{\ell}}{2\sin\frac{\pi}{\ell}-\sqrt{2(1-\cos\frac{\pi}{\ell})}}|\ln\va|+o(|\ln\va|),\quad\rho_\va=
\frac{2\sin\frac{\pi}{\ell}}{2\sin\frac{\pi}{\ell}-\sqrt{2(1-\cos\frac{\pi}{\ell})}}|\ln\va|+o(|\ln\va|).
 $$
 Moreover, as $\va\to 0^+$,
 $$
\|u^\va(\cdot)-v^\va(T_\ell\cdot)\|_{H^1}+\|u^\va(\cdot)-v^\va(T_\ell\cdot)\|_{L^\infty}\to
0.
 $$
 Here $T_\ell\in SO(3)$ is the rotation on the $(x_1,x_2)$ plane of
 $\frac{\pi}{\ell}$.
\end{theorem}

Theorem~\ref{main} says that $|x^i_\va-y^j_\va|/|\ln\va|\to
a_{i,j}>0$ ($i,j=1,\cdots,\ell$) as $\va\to 0$. Hence
Theorem~\ref{main} gives  segregated types of solutions for system
\eqref{eqmain} with the essential support of the two components
being  segregated for $\va$ sufficiently small.

We also construct segregated vector solutions for the following
three coupled systems,  which arise when one considers the
propagation of pulses in a $3$-core couplers with circular symmetry:
\begin{equation}\label{eqmain2}
\left\{
\begin{array}{ll}
-\Delta u+u=u^3-\varepsilon (v+\omega),\hspace{1cm}& x\in \R^3, \vspace{0.2cm}\\
-\Delta v+v=v^3 -\varepsilon (u+\omega),& x\in\R^3,\vspace{0.2cm}\\
-\Delta \omega+\omega=\omega^3 -\varepsilon (u+v),& x\in\R^3.
\end{array}
\right.
\end{equation}

\medskip
\begin{theorem} \label{main2}
For any integer $\ell\ge 2$, there exists $\varepsilon_0$ such that
for $\va\in (0,\va_0)$, problem~\eqref{eqmain2} has a solution
$(u^\va,v^\va,\omega^\va)\in (H^1(\R^3))^3$ satisfying
$$
u^\va\sim \sum\limits_{j=1}^\ell U_{x_\va^j},\quad v^\va\sim
\sum\limits_{j=1}^\ell U_{y_\va^j},\quad \omega^\va\sim U,
$$
 where $x_{\va}^j$ and $y_{\va}^j$ are the same as those of Theorem~\ref{main} if $\ell>2$, but for $\ell=2$
 $$
r_\va= |\ln\va|+o(|\ln\va|),\quad\rho_\va= |\ln\va|+o(|\ln\va|).
 $$
  Moreover, as $\va\to 0^+$,
 $$
\|u^\va(\cdot)-v^\va(T_\ell\cdot)\|_{H^1}+\|u^\va(\cdot)-v^\va(T_\ell\cdot)\|_{L^\infty}\to
0.
 $$
\end{theorem}

\medskip
\begin{remark} The segregation  nature of these
solutions are demonstrated from the $L^\infty$ estimates in the
theorems and will be more clear in Propositions \ref{pro2} and
\ref{pro41} stated later after we find a good approximate solution
and fix the notations. Roughly speaking, as $\va\to 0$, the
segregated solutions may have a large number of bumps near infinity
while the locations of the bumps for $u$ and $v$ have an angular
shift.
\end{remark}

\begin{remark}
In \cite{ACR}, to guarantee  the existence of the solutions, the
side of the polytope should be greater than the radius, which
implies that the number of the solutions cannot be very large (at
least in two dimensional case). In our results, the number of the
bumps can be very large, and the energy of the solutions can become
so large as we expected. Moreover,  all the bumps are positive,
which implies that these solutions bifurcate from the symmetric state at $\va=0$. Hence our
results are in striking contrast with those of \cite{ACR}.
\end{remark}

\begin{remark}
Our argument also works well for the following more general problems
in various dimensional case
$$
\left\{
\begin{array}{ll}
-\Delta u_j+u_j=|u_j|^{p-2}u_j-\va\sum\limits_{i\neq j}^N u_i,\hspace{1cm}& x\in \R^d, \vspace{0.2cm}\\
u_j\in H^1(\R^d),\quad j=1,\cdots,N.
\end{array}
\right.
$$
Here $N=2,3$,  $d>1$, and $2<p<2^*$, where $2^*=2d/(d-2)$ if $d\ge
3$ and $ 2^*=+\infty$ if $d=2$. We point out that our results are
most likely wrong for $d=1$, which is verified by the numerical
computation in \cite{AA}.
\end{remark}

To prove the main results, we will employ the well-known
Lyapunov-Schmidt reduction (see, e.g., \cite{R}) to glue the
functions $U_{x^j}$ (or $U_{y^j}$) ($j=1,\cdots, \ell$). In
performing this technique,  to find critical points of the reduced
functionals, a basic requirement is that the error terms of the
functionals, which come from the finite dimensional reduction,
should be of higher order small data of the main terms in the
reduced functionals. However, in our linearly coupled systems,
different from the nonlinearly coupled ones (see, e.g., \cite{LtW3}
and \cite{PW}), if we choose $(U,U)$ as an approximate  solution,
the error terms from the linear coupling  dominate the main terms
(which are generated from the interaction between the neighbor
bumps) of the reduced functionals. To overcome this difficulty, we
should modify another  approximate solution $(U,0)$. This idea is
essentially from \cite{ACR}, where an approximate solution
$(U_\va,V_\va)$ bifurcates from $(U,0)$. However, comparing with
\cite{ACR}, we encounter  two more problems. Firstly, we need a new
approximate solution and a precise  estimate on it. To this end, we
will make a  modification on  $(U,0)$ carefully  by using the
reduction technique (see section 2). This procedure provides us a
more accurate approximate solution with required estimate. Secondly,
after performing  a second reduction, we need to solve a
two-dimensional critical point problem, which requires us to choose a
very delicate domain and make a precise analysis on the reduced
functionals. So, we  need a very accurate estimate on the energy of
the reduced functional, which also needs the help of the approximate
solution. Hence, here we will perform the reduction twice and deal
with more complicated reduced functionals.

To find vector solutions with two components having the prescribed
number of bumps, we will employ the idea proposed by Wei and Yan in
\cite{WY}, where infinitely many positive solutions were constructed
for single Schr\"{o}dinger equations. This idea is also effective in
finding infinitely many  non-radial positive solutions for
semilinear elliptic problems with  critical or super-critical
Sobolev growth (see, for example, \cite{WY1,WY2,WY3}) and
Schr\"{o}dinger systems with nonlinear coupling (see, for example,
\cite{PW}).

This paper is organized as follows. In section~2, we will perform a
reduction argument for the first time and modify the vector function
$(U,0)$  so that we can get an accurate approximate solution and a
precise estimate on it. In section~3, using the approximate
solution, we will formulate a more precise version of the main
results which give more precise descriptions about the segregated
character of the solutions. We will also carry out the reduction for
the second time to a finite two-dimensional setting and prove
Theorem~\ref{main}. The study of existence of segregated solutions
for a system coupled by three nonlinear Schr\"{o}dinger equations will
be briefly discussed in section~4  by using our framework of
methods. We conclude with the energy expansion in the appendix.


\section{An approximate solution}

\setcounter{equation}{0} In this section,  to  look for a proper
approximate vector solution, we need to modify $(U,0)$. Let
$H^1_r(\R^3)$ and $L^2_r(\R^3)$ denote the corresponding spaces of
radial functions. For $(u,v)\in H^1_r(\R^3)\times H^1_r(\R^3)$, we
define $\|(u,v)\|=\|u\|_{H^1(\R^3)}+\|v\|_{H^1(\R^3)}$.

 Solving in $H^1_r(\R^3)\times H^1_r(\R^3)$ the equations
$$
\left\{
\begin{array}{ll}
-\Delta u_1+u_1-3U^2u_1=0,\hspace{1cm}& x\in \R^3, \vspace{0.2cm}\\
-\Delta \tilde v_1+\tilde v_1=-U,& x\in\R^3,
\end{array}
\right.
$$
we get $u_1=0$ and $\tilde v_1<0$. Let $c(x)\in H^1_r(\R^3)$ satisfy
$$
-\Delta c(x)+c(x)=\tilde  v_1^3,
$$
then $v_1=\tilde v_1+\va^2 c(x)$ solves
$$
-\Delta v_1+v_1=-U+\va^2\tilde v_1^3.
$$

Now for $k\ge 2$, by the Fredholm Alternative Theorem   we can
define $(u_k,v_k)\in H^1_r(\R^3)\times H^1_r(\R^3)$ by solving
\begin{equation}\label{eq01}
\left\{
\begin{array}{ll}
-\Delta u_k+u_k-3U^2u_k=-kv_{k-1},\hspace{1cm}& x\in \R^3, \vspace{0.2cm}\\
-\Delta v_k+v_k=-k u_{k-1},& x\in\R^3.
\end{array}
\right.
\end{equation}
We can also see that $v_2=0$.

\begin{remark}
Here we execute the second modification  by defining $v_1$ so that
the norm of the error terms in $H^1(\R^3)\times H^1(\R^3)$ can be
dominated by $C\va^4$ (see Proposition~\ref{pro2.1} later).
\end{remark}

We want to find suitable $(w(x),h(x))\in H^1_r(\R^3)\times
H^1_r(\R^3)$ such that
\begin{equation}\label{eq02}
(U_\va,\,v_\va)=:\Bigl(U+\sum\limits_{i=1}^4
\frac{\va^i}{i!}u_i+\va^4 w,\,\sum\limits_{i=1}^4
\frac{\va^i}{i!}v_i+\va^4 h\Bigr)
\end{equation}
solves  problem~\eqref{eqmain}.

Inserting \eqref{eq02} into \eqref{eqmain} and employing
\eqref{eq01}, we find
\begin{equation}\label{eq04}
\left\{
\begin{array}{ll}
-\Delta
w+w-3U^2w=\dis\frac{H_{\va}(u_2,u_3,u_4,v_4,U)}{\va^4}+l_\va(h,w)+\frac{R_\va(\va^4w)}{\va^4},
 \vspace{0.2cm}\\
-\Delta h+h=\dis \frac{\bar
H_\va(v_1,v_3,v_4,u_4)}{\va^4}+\bar{l}_\va(h,w)+\dis\frac{\bar{R}_\va(\va^4
h)}{\va^4},
\end{array}
\right.
\end{equation}
where
\begin{eqnarray*}
&&H_{\va}(u_2,u_3,u_4,v_4,U)=\Bigl(U+\dis\sum\limits_{i=2}^4
\frac{\va^i}{i!}u_i\Bigr)^3-U^3-3U^2\sum\limits_{i=2}^4
\frac{\va^i}{i!}u_i-\frac{\va^{5}}{4!}v_4,\\
&&l_\va(w,h)=3\Bigl(\Bigl(U+\dis\sum\limits_{i=2}^4
\frac{\va^i}{i!}u_i\Bigr)^{2}-U^{2}\Bigr)w-\va h,\\
&&R_\va(\va^4 w)= 3 \Bigl(U+\sum\limits_{i=2}^4
\frac{\va^i}{i!}u_i\Bigr)(\va^4 w)^2+(\va^4 w)^3, \\
&&\frac{\bar
H_\va(v_1,v_3,v_4,u_4)}{\va^4}=\Bigl(\sum\limits_{i=1}^4\dis\frac{\va^{i}}{i!}v_i\Bigr)^3
-\va^3\tilde v_1^3-\frac{\va^5}{4!}u_4,\\
 &&\bar{l}_\va=-\va w+3\Bigl(\dis\sum\limits_{i=1}^4
\frac{\va^i}{i!}v_i\Bigr)^{2}h,\\
&&\bar{R}_\va(\va^4 h)=3\Bigl(\dis\sum\limits_{i=1}^4
\frac{\va^i}{i!}v_i\Bigr)(\va^4 h)^2+(\va^4h)^3.
\end{eqnarray*}

Direct calculation yields that
\begin{eqnarray*}
 \Bigl|\frac{H_{\va}(u_2,u_3,u_4,v_4,U)}{\va^4}\Bigl|
&\le&C,\\
\frac{ \bar
H_\va(v_1,v_3,v_4,u_4)}{\va^4}&=&\Bigl(\Bigl(\sum\limits_{i=1}^4\dis\frac{\va^{i}}{i!}v_i\Bigr)^3
-\va^3\tilde v_1^3-\frac{\va^5}{4!}u_4\Bigr)\Bigl/\va^4,\\
&=&\Bigl(\Bigl(\sum\limits_{i=1}^4\dis\frac{\va^{i}}{i!}v_i\Bigr)^3-\va^3v_1^3+\va^3(v_1^3
-\tilde v_1^3)-\frac{\va^5}{4!}u_4\Bigr)\Bigl/\va^4\\
&=&O(\va),
\end{eqnarray*}
where we have used the fact $v_2=0$ and $|v_1-\tilde
v_1|=O(\va^{2})$.

Since the kernel of operator
$$
\mathbb{L} \left(\begin{array}{ll}  w\vspace{0.2cm}\\
h
\end{array}
\right) =\left(\begin{array}{ll}  -\Delta
w+w-3U^2w\vspace{0.2cm}\\
-\Delta h+h
\end{array}
\right) :\,\,H^1_r(\R^3)\times H^1_r(\R^3)\to L^2_r(\R^3)\times
L^2_r(\R^3)
$$
is $\{(0,0)\}$ in $H^1_r(\R^3)\times H^1_r(\R^3)$, we know that the
operator $\mathbb{L}$ has bounded inverse in $H^1_r(\R^3)\times
H^1_r(\R^3)$.

Define
$$
 \left(\begin{array}{ll}  \bar w\vspace{0.2cm}\\
\bar h
\end{array}
\right) =\mathbb{L}^{-1}\left(\begin{array}{ll}
\dis\frac{H_{\va}(u_2,u_3,u_4,v_4,U)}{\va^4}+l_\va(h,w)
+\frac{R_\va(\va^4w)}{\va^4}\vspace{0.2cm}\\
\dis\frac{\bar
H_{\va}(v_1,v_3,v_4,u_4)}{\va^4}+\bar{l}_\va(h,w)+\frac{\bar{R}_\va(\va^4
h)}{\va^4}
\end{array}
 \right)
 =:\mathbb{A}\left(\begin{array}{ll}  w\vspace{0.2cm}\\
 h
\end{array}
\right)
$$
and the set
$$
\mathbb{S}=\{(w,h)\in H^1_r(\R^3)\times
H^1_r(\R^3):\,\,\|(w,\,h))\|\le |\va|^{-\sigma}\},
$$
where $\sigma>0$ is sufficiently  small.

Then by direct calculation, we find for $(w,h),\, (w_1,h_1),\,
(w_2,h_2)\in \mathbb{S}$,
\begin{eqnarray*}
&&\|(\bar w,\,\bar h)\|\le C(1+\va)\le |\va|^{-\sigma},\\
&&\|(\bar w_1-\bar w_2, \bar h_1-\bar
h_2)\|=\|\mathbb{A}(w_1-w_2,h_1-h_2)\|\\
&&\hspace{3.9cm}\le |\va| \|(w_1-w_2,h_1-h_2)\|<\frac12
\|(w_1-w_2,h_1-h_2)\|.
\end{eqnarray*}
 Therefore,    the operator $\mathbb{A}$ maps $\mathbb{S}$ into $\mathbb{S}$ and is a contraction map. So,
by the contraction mapping theorem, there exists $(w,h)\in
\mathbb{S}$, such that $(w,h)=\mathbb{A}(w,h)$. Direct computation
yields
\begin{eqnarray*}
&&\Bigl|\int_{\R^3} \frac{H_{\va}(u_2,u_3,u_4,v_4,U)}{\va^4}
\varphi+\dis\frac{\bar
H_{\va}(v_1,v_3,v_4,u_4)}{\va^4}\psi\Bigr|\\
&\le& C\|(\varphi,\psi)\|,\,\,\forall\,\,(\varphi,\psi)\in
H^1_r(\R^3)\times H^1_r(\R^3).
\end{eqnarray*}
As a result, we see
\begin{equation}\label{eq07}
\|(w,h)\|\le C\va^4.
\end{equation}

Now we  consider the asymptotic behavior of $u_i,v_i,
(i=1,\cdots,4)$ at infinity. We claim that for any fixed small
$\tau>0$, there exists a positive  constant $C$ depending on $\tau, u_i,v_i,
(i=1,\cdots,4)$ such that
\begin{equation}\label{eq09}
|u_i(r)|+|v_i(r)|\le C e^{-(1-\tau)r}, \,\,(i=1,\cdots,4),\,\,
\,\forall\,\,r>1.
\end{equation}

Indeed,  by induction, we  suppose $|v_{i-1}|\le C_{i-1}
e^{-(1-\tau)r}$. Since
\begin{eqnarray*}
&&-\Delta
e^{-(1-\tau)r}+e^{-(1-\tau)r}-3U^{2}e^{-(1-\tau)r}\\
&=&\Bigl(1-(1-\tau)^2+\frac{N-1}{r}-3U^{2}\Bigr)e^{-(1-\tau)r},
\end{eqnarray*}
we can choose $\bar C_{i}, R_i$ depending on $u_i,\tau, i$ and
$C_{i-1}$ such that $\bar C_ie^{-(1-\tau)r}$ is a super-solution of
the first equation of \eqref{eq01} on $\R^3\setminus B_{R_i}(0)$. By
comparison theory of elliptic equations, we conclude
$$
u_i\le \bar C_ie^{-(1-\tau)r},\,\,\forall\,\,r\ge R_i.
$$
With the same argument, we can also prove that
$$
u_i\ge -\bar C_ie^{-(1-\tau)r},\,\,\forall\,\,r\ge R_i.
$$
Hence, we can choose $C_i$ depending on $u_i,\tau,i,C_{i-1}$ such
that
$$
|u_i(r)|\le  C_ie^{-(1-\tau)r},\,\,\forall\,\,r>1.
$$

Similarly, we can prove that $|\tilde v_1|\le Ce^{-(1-\tau)r}$,
$|c(x)|\le Ce^{-(1-\tau)r}$ and also $|v_i|\le C_i e^{-(1-\tau)r}$
for $r>1$.

The above results can be summarized as
\begin{proposition}\label{pro2.1}
 There exists $\varepsilon_0>0$ such that for $\varepsilon \in
(-\varepsilon_0,\,\varepsilon_0)$, problem \eqref{eqmain} has a
solution $(U_\varepsilon,\,v_\varepsilon)\in H^1_r(\R^3)\times
H^1_r(\R^3)$ satisfying $U_\varepsilon\to U$, $v_\varepsilon \to 0$
in $H^1_r$ as $\varepsilon \to 0$. Moreover,
\begin{equation}\label{eq 22}
\begin{array}{ll}
U_\varepsilon= U+
\sum\limits_{i=2}^{4}\dis\frac{\varepsilon^{i}}{i!}u_{i}+w,
\,\,\,v_\varepsilon=
\sum\limits_{i=1}^{4}\dis\frac{\varepsilon^{i}}{i!}v_{i}+h.
\end{array}
\end{equation}
Here
\begin{equation}\label{eq 23}
\|(w,h)\|\le \tilde C\va^4,
\end{equation}
$\tilde C>0$ is independent of $\va$. $u_i$ and $v_i$ satisfy
\begin{equation}\label{eq23}
|u_i(r)|+|v_i(r)|\le  C e^{-(1-\tau)r}, \,\, \,\forall\,\,r>1,
\end{equation}
where $\tau>0$ is any fixed small constant, $ C$ depends on $\tau,
u_i, v_i, \,(i=1,\cdots,4)$.
\end{proposition}

With the same argument we can also construct a solution for
problem~\eqref{eqmain2} which is linearly coupled by three
equations.

The main result is
\begin{proposition}\label{pro2.2}
 There exists $\varepsilon_0>0$ such that for $\varepsilon \in
(-\varepsilon_0,\,\varepsilon_0)$, problem \eqref{eqmain2} has a
solution $(U_\varepsilon,\,v_\varepsilon,\omega_\va)\in
(H^1_r(\R^3))^3$ satisfying $U_\varepsilon\to U$, $v_\varepsilon \to
0$ and $\omega_\varepsilon \to 0$ in $H^1_r$ as $\varepsilon \to 0$.
Moreover,
\begin{equation}\label{eq 25}
\begin{array}{ll}
U_\varepsilon= U+
\sum\limits_{i=2}^{4}\dis\frac{\varepsilon^{i}}{i!}u_{i}+w,
\,\,\,v_\varepsilon=
\sum\limits_{i=1}^{4}\dis\frac{\varepsilon^{i}}{i!}v_{i}+h,\,\,\,\omega_\varepsilon=
\sum\limits_{i=1}^{4}\dis\frac{\varepsilon^{i}}{i!}\omega_{i}+g.
\end{array}
\end{equation}
Here
\begin{equation}\label{eq 26}
\|(w,h,g)\|=:\|w\|_{H^1(\R^3)}+\|h\|_{H^1(\R^3)}+\|g\|_{H^1(\R^3)}\le
\bar C\va^4,
\end{equation}
$\bar C>0$ is independent of $\va$. $u_i,\,v_i$ and $\omega_i$ satisfy
\begin{equation}\label{eq231}
|u_i(r)|+|v_i(r)|+|\omega_i(r)|\le  C e^{-(1-\tau)r}, \,\,
\,\forall\,\,r>1,
\end{equation}
where $\tau>0$ is any fixed small constant, $ C$ depends on $\tau,
u_i, v_i, \omega_i\,(i=1,\cdots,4)$.
\end{proposition}
\begin{proof}
Solve
$$
\left\{
\begin{array}{ll}
-\Delta u_1+u_1-3U^2u_1=0,\hspace{1cm}& x\in \R^3, \vspace{0.2cm}\\
-\Delta \tilde v_1+\tilde v_1=-U,& x\in\R^3,\vspace{0.2cm}\\
-\Delta \tilde \omega_1+\tilde \omega_1=-U,& x\in\R^3,
\end{array}
\right.
$$
then $u_1=0$, $\tilde v_1\in H^1_r(\R^3)$, $\tilde \omega_1 \in
H^1_r(\R^3)$. Let $c(x),\,d(x)\in H^1_r(\R^3)$ satisfy
$$
-\Delta c(x)+c(x)=\tilde  v_1^3,\,\,-\Delta d(x)+d(x)=\tilde
\omega_1^3,
$$
we see that $v_1=\tilde v_1+\va^2 c(x)$ and $\omega_1=\tilde
\omega_1+\va^2 d(x)$ solve
$$
-\Delta v_1+v_1=-U+\va^{2}\tilde v_1^3,\,\,\,-\Delta
\omega_1+\omega_1=-U+\va^{2}\tilde \omega_1^3.
$$

For $k\ge 2$, we can define $(u_k,v_k,\omega_k)\in (H^1_r(\R^3))^3$
by solving
\begin{equation}\label{eq011}
\left\{
\begin{array}{ll}
-\Delta u_k+u_k-3U^2u_k=-k(v_{k-1}+\omega_{k-1}),\hspace{1cm}& x\in \R^3, \vspace{0.2cm}\\
-\Delta v_k+v_k=-k (u_{k-1}+\omega_{k-1}),& x\in\R^3,\vspace{0.2cm}\\
-\Delta \omega_k+\omega_k=-k (u_{k-1}+v_{k-1}),& x\in\R^3.
\end{array}
\right.
\end{equation}
Proceeding as we prove Proposition~\ref{pro2.1}, we can find
$(w,h,g)\in (H^1_r(\R^N))^3$ such that \eqref{eq 26} and
\eqref{eq231} hold true and $(U_\va,v_\va,\omega_\va)$ defined by
\eqref{eq 25} satisfies problem \eqref{eqmain2}.
\end{proof}

\section{ Segregated vector solutions for 2 coupled Schr\"{o}dinger  system }
We will use $(U_\varepsilon,v_\varepsilon)$ to construct multi-bump
solutions for \eqref{eqmain}. It follows from
Proposition~\ref{pro2.1} that $(U_\va,v_\va)$ has the form
\begin{equation}\label{eq 22}
\begin{array}{ll}
U_\varepsilon= U+ \va^2 p_\va(r)+w, \,\,\,v_\varepsilon= \va
q_\va(r)+h,
\end{array}
\end{equation}
where
$$
p_\va(r)\le Ce^{-(1-\tau)r},\,\,q_\va(r)\le
Ce^{-(1-\tau)r},\,\,\,\|(w,h)\|\le C\va^4.
$$
Here $C$ is independent of $\va$,  and $\tau>0$ is defined in \eqref{eq23}.

For any integer $\ell\ge 2$, set
$$
m=2\sin\frac{\pi}{\ell},\,\,n=\sqrt{2\Bigl(1-\cos\frac{\pi}{\ell}\Bigr)}.
$$
Then it can be easily check that
$$
m>n>0,\quad 2<\frac{m}{m-n}<4.
$$
 Let $x^j$ and $y^j$ be defined by \eqref{1-29-9}
and \eqref{1-29-10} respectively. In this section,  we assume
\begin{equation}\label{1-20-5}
(r,\rho)\in \mathcal{D}_\va\times  \mathcal{D}_\va =:  \Bigl[ \frac
{|\ln\va|}{m-n+\frac{\mu\ln|\ln\va|}{|\ln\va|}},\,\frac
{|\ln\va|}{m-n} \Bigr]\times \Bigl[ \frac
{|\ln\va|}{m-n+\frac{\mu\ln|\ln\va|}{|\ln\va|}},\,\frac
{|\ln\va|}{m-n} \Bigr],
\end{equation}
where the constant $\mu>m-n$. For any function $W:\R^3\to\R$  and
$\xi\in\R^3$, we define $W_{\xi}=W(x-\xi)$. Set
$$
U_{\va,r}=\sum\limits_{i=1}^\ell
U_{\va,x^i},\,\,v_{\va,r}=\sum\limits_{i=1}^\ell v_{\va,x^i},\,\,
U_{\va,\rho}=\sum\limits_{i=1}^\ell
U_{\va,y^i},\,\,v_{\va,\rho}=\sum\limits_{i=1}^\ell v_{\va,y^i},
$$
and
\[
 Y_{\va,j}=\frac{\partial U_{\va,x^j}}{\partial r},\,\,
Z_{\va,j}=\frac{\partial U_{\va,y^j}}{\partial \rho},\quad
j=1,\cdots,\ell.
  \]

Define
\[
\begin{split}
H_{s}=\bigl\{ u: \,& u\in H^1(\R^3), u\;\text{is even in} \;x_h, h=2,3,\\
& u(r\cos\theta , r\sin\theta, x')= u(r\cos(\theta+\frac{2\pi
j}\ell) , r\sin(\theta+\frac{2\pi j}\ell), x') \bigr\},
\end{split}
\]
and
\begin{equation}\label{E}
 \mathbb{E}=\Bigl\{ (u,v)\in H_{s}\times H_{s},\;
\sum\limits_{j=1}^\ell\int_{\R^3} U_{\va,x^j}^2  Y_{\va,j}
u=0,\,\,\sum\limits_{j=1}^\ell\int_{\R^3} U_{\va,y^j}^{2} Z_{\va,j}
v=0 \Bigr\}.
\end{equation}

To prove Theorem~\ref{main}, it suffices to prove
\begin{proposition}\label{pro2}
For any integer $\ell\ge 2$, there exists $\varepsilon_0>0$ such
that for $\va\in (0,\,\va_0)$, problem~\eqref{eqmain} has a solution
$(u,v)$ with the form
$$
u=U_{\va,r}+v_{\va,\rho}+\varphi_\va,\,\,v=U_{\va,\rho}+v_{\va,r}+\psi_\va,
$$
where $(\varphi_\va,\psi_\va)\in \mathbb{E}$ satisfies
$\|(\varphi_\va,\,\psi_\va)\|=o(\va^{\frac{m}{m-n}})$.
\end{proposition}

Let
\[
\begin{array}{ll}
I(u,v)=&\dis\frac12\int_{\R^3} \bigl( |\nabla u|^2+ u^2+ |\nabla
v|^2+
v^2\bigl)\vspace{0.2cm}\\
&-\dis\frac1{4}\int_{\R^3}\bigl(u^{4}+v^{4}\bigl)+\va\int_{\R^3}uv,\quad
(u,v)\in H_s\times H_s,
\end{array}
\]
and
\[
J(\varphi,\psi)=
I(U_{\va,r}+v_{\va,\rho}+\varphi,U_{\va,\rho}+v_{\va,r}+\psi).
\]

Expand $J(\varphi,\psi)$ as follows:
\begin{equation}\label{1-3-10}
J(\varphi,\psi)=J(0,0)-l(\varphi,\psi)+ \frac12 \bar L(\varphi,\psi)
-R(\varphi,\psi), \quad (\varphi,\psi)\in \mathbb{E},
\end{equation}
where
\[
\begin{array}{ll}
l(\varphi,\psi)=&\dis\int_{\R^3}
\bigl((U_{\va,r}+v_{\va,\rho})^3-\sum\limits_{j=1}^\ell
U_{\va,x^j}^3 -\sum\limits_{j=1}^\ell
v_{\va,y^j}^3\bigr)\varphi\vspace{0.2cm}\\
&+\dis\int_{\R^3}
\bigl((U_{\va,\rho}+v_{\va,r})^3-\sum\limits_{j=1}^\ell
U_{\va,y^j}^3 -\sum\limits_{j=1}^\ell v_{\va,x^j}^3\bigr)\psi
\end{array}
\]
\[
\begin{split}
\bar L(\varphi,\psi)=&\int_{\R^3} \bigl(|\nabla
\varphi|^2+\varphi^2-
 3(U_{\va,r}+v_{\va,\rho})^{2} \varphi^2\bigr)\vspace{0.2cm}\\
 &+\int_{\R^3} \bigl(|\nabla \psi|^2+\psi^2-
 3(U_{\va,\rho}+v_{\va,r})^{2} \psi^2\bigr)+2\va\int_{\R^3}\varphi\psi,
\end{split}
\]
and
$$
R(\varphi,\psi)=\int_{\R^3}
\bigl((U_{\va,\rho}+v_{\va,r})\varphi^3+(U_{\va,\rho}+v_{\va,r})\psi^3\bigr)+\frac{1}4\int_{\R^N}(\varphi^4+\psi^4).
$$

\begin{remark}\label{re3.2}
Here, in the expression of the linear part $l(\varphi,\psi)$, there
are no terms from the coupled term $\va\int_{\R^3} uv$ since we use
 $(U_\va,\,v_{\va})$ to construct the vector solutions. We will see
later in the proof of Proposition~\ref{pro2} that  this choice of
the approximate  solution guarantees that the error terms of the
reduced functional are dominated by $\va^{\frac{m+\sigma}{m-n}}$,
which is of higher order small datum of the main terms. However, if
we use $(U, U)$ as an approximate solution, then in the expression
of $l(\varphi,\psi)$, the terms from the coupling like
$\va\int_{\R^N}(\sum_{j=1}^\ell
U_{x^j})\psi+\va\int_{\R^3}(\sum_{j=1}^\ell U_{y^j})\varphi$ will
appear, which implies $\|l(\varphi,\psi)\|=O(\va)$. Hence the error
terms of the reduced functional are of order $O(\va^2)$, which will
dominate the main terms and we have no way to solve the reduced
functional.
\end{remark}

It is easy to check that $\bar L(\varphi,\psi)$ can be generated  by
a bounded linear operator $L$ from $\mathbb{E}$ to $\mathbb{E}$,
which is defined as
\begin{eqnarray*}
\bigl\langle L(u, v),(\varphi,\psi)\bigr\rangle &=&\int_{\R^3}
\bigl(\nabla u \nabla \varphi+u\varphi-
 3(U_{\va,r}+v_{\va,\rho})^{2}u \varphi\bigr)\\
 &&+\int_{\R^3} \bigl(\nabla v\nabla \psi+v\psi-
 3(U_{\va,\rho}+v_{\va,r})^{2} v\psi\bigr)+\va\int_{\R^3}(u \psi+v\varphi).
\end{eqnarray*}

Now,  we discuss the  invertibility of $L$.
\begin{lemma}\label{2.2}
 There exists $\va_0>0$, such that for $\va\in (0,\,\va_0)$, there is a constant $\varrho>0$, independent of $\va$, satisfying
  that for any
 $(r,\rho)\in  \mathcal {D}_\va\times \mathcal {D}_\va$,
\[
 \|L(u,v)\|\ge \varrho\|(u,v)\|,\quad  (u,v)\in \mathbb{E}.
 \]
\end{lemma}

\begin{proof}
 Suppose to the contrary that there are $\va_n\to 0^+$ (as $n\to+\infty$),
$(r_{n},\rho_{n})\in  \mathcal {D}_{\va_n}\times \mathcal
{D}_{\va_n}$, and $(u_{n},v_{n})\in \mathbb{E}$, with
\begin{equation}\label{1-l21}
\bigl\langle L (u_{n},v_{n}),(\varphi,\psi)\bigr\rangle = o_n(1)
\|(u_{n},v_{n})\|\|(\varphi,\psi)\|,\quad \forall\;
(\varphi,\psi)\in \mathbb{E}.
\end{equation}
We may assume that $\|(u_n,v_n)\|=1$. We see from \eqref{1-l21},
\begin{equation}\label{22}
\begin{array}{ll}
&\dis\int_{\R^3} \bigl(\nabla u_n \nabla \varphi+u_n\varphi-
 3(U_{\va_n,r_n}+v_{\va_n,\rho_n})^{2}u_n \varphi\bigr)\vspace{0.2cm}\\
 &+\dis\int_{\R^3} \bigl(\nabla v_n\nabla \psi+v_n\psi-
 3(U_{\va_n,\rho_n}+v_{\va_n,r_n})^{2} v_n\psi\bigr)\vspace{0.2cm}\\
 &+\va_n\dis\int_{\R^3}(u_n \psi+v_n\varphi)=o_n(1)\|(\varphi,\psi)\|,\,\,\forall\,\,(\varphi,\psi)
\in \mathbb{E}.
\end{array}
\end{equation}
In particular,
\begin{equation}\label{24}
\begin{array}{ll}
 & \dis\int_{\R^3} \bigl(|\nabla u_n|^2+u_n^2-
 3(U_{\va_n,r_n}+v_{\va_n,\rho_n})^{2}u_n^2\bigr)\vspace{0.2cm}\\
 &+ \dis\int_{\R^3} \bigl(|\nabla v_n|^2+v_n^2-
 3(U_{\va_n,\rho_n}+v_{\va_n,r_n})^{2}v_n^2\bigr)+2\va_n\dis\int_{\R^3}u_nv_n
 =o_n(1),
\end{array}
\end{equation}
and
\begin{equation}\label{0-19-3}
\dis\int_{\R^3} \bigl(|\nabla u_n|^2+u_n^2+|\nabla
v_n|^2+v_n^2\bigr)=1.
\end{equation}

 Let
 $$ \bar u_n(x)= u_n(x-x^1),\,\,\,\bar v_n(x)= v_n(x-y^1).$$

We may assume the existence of $u$, such that as $n\to +\infty$,
\[
\bar u_n \to  u, \quad \text{weakly in}\;
H^1_{loc}(\R^3),\hspace{1cm} \bar u_n \to  u, \quad \text{strongly
in}\; L^2_{loc}(\R^3).
\]
Moreover, $ u$ is even in $x_h$, $h=2,3.$

By symmetry, we see
$$
\int_{\R^3}U_{\va_n,x_n^1}^{2}Y_{\va_n,1}u_n=0.
$$
It follows from $(U_{\va_n}, v_{\va_n})\to (U,0)$ in
$H^1(\R^3)\times H^1(\R^3)$ that
\begin{eqnarray*}
&&\Bigl|\int_{\R^3}U_{\va_n,x_n^1}^{2}Y_{\va_n,1}u_n-\int_{\R^3}U_{x_n^1}^{2}\frac{\partial
U_{x_n^1}}{\partial r_n}u_n\Bigl|\\
&\le&
\Bigl|\int_{\R^3}(U_{\va_n,x_n^1}^{2}-U_{x_n^1}^{2})Y_{\va_n,1}u_n\Bigl|
+\Bigl|\int_{\R^3}U_{x_n^1}^{2}\Bigl(Y_{\va_n,1}-\frac{\partial
U_{x_n^1}}{\partial r_n}\Bigr)u_n\Bigl|\to 0,
\,\,\,(n\to+\infty).
\end{eqnarray*}
Hence
$$
\int_{\R^3}U_{x_n^1}^{2}\frac{\partial U_{x_n^1}}{\partial
r_n}u_n\to 0,
$$
which implies
\begin{equation}\label{5-l21}
\int_{\R^3} U^{2} \frac{\partial U}{\partial x_1}  u =0.
\end{equation}

Let $ \varphi\in C_0^\infty(B_R(0))$  be even in $x_h$, $h=2,3$.
Define $ \varphi_n(x)=:  \varphi (x-x^1)\in C_0^\infty(B_R(x^1))$.
We may identify $\varphi_n(x)$ as elements in $H_s$ by redefining
the values outside $B_R(x^1)$ with the symmetry.

From the fact that $U_{\va_n}\to U$ and $v_{\va_n}\to 0$ in
$H^1(\R^3)$, we deduce
\begin{equation}\label{eq200}
\begin{array}{ll}
&\dis\int_{\R^3}(U_{\va_n,r_n}+v_{\va_n,\rho_n})^{2}u_n
\varphi_n\vspace{0.2cm}\\
=&\dis\int_{\R^3}(U^2_{\va_n,r_n}+2U_{\va_n,r_n}v_{\va_n,\rho_n}+v_{\va_n,\rho_n}^{2})u_n
\varphi_n=\dis\int_{\R^3}U^2u\varphi +o_n(1).
\end{array}
\end{equation}

 Then  choosing $( \varphi, \psi)=(\varphi_n,0)$ in \eqref{22} and
 considering \eqref{eq200},
we can use the argument in \cite{WY}, to prove that $u$ solves
\begin{equation}\label{eq2}
-\Delta u+u-3 U^{2}u=0,\hspace{1cm}x\in \R^3.
\end{equation}

 Since we
work in the space of functions which are even in $x_2$ and $x_3$, we
see $u=c\frac{\partial U}{\partial x_1}$ for some $c$, which implies
that $ u=0$ since $ u$ satisfies \eqref{5-l21}.

To deal with $v_n$, we first claim  that for any  $v(x)\in H_s$,
$v(x)$ is even with respect to the ray with an angle of $\pi/\ell$.

Indeed, suppose  that $|(x_1,x_2)|=a$, then
\begin{eqnarray*}
v(x)&=:&v(a\cos
(\frac{\pi}{\ell}+\theta),a\sin(\frac{\pi}{\ell}+\theta),x_3)\\
&=&v(a\cos
(\frac{\pi}{\ell}+\theta),-a\sin(\frac{\pi}{\ell}+\theta),x_3)\\
&=&v(a\cos
(-\frac{\pi}{\ell}-\theta),a\sin(-\frac{\pi}{\ell}-\theta),x_3)\\
&=&v(a\cos
(\frac{\pi}{\ell}-\theta),a\sin(\frac{\pi}{\ell}-\theta),x_3).
\end{eqnarray*}

Now as we deal with $u_n$,  we can check
\[
\bar v_n \to  0, \quad \text{weakly in}\;
H^1_{loc}(\R^3),\hspace{1cm} \bar v_n \to  0, \quad \text{strongly
in}\; L^2_{loc}(\R^3).
\]

Similar to \eqref{eq200},  using the fact that $U_{\va_n}\to U$
and $v_{\va_n}\to 0$ in $H^1(\R^3)$ as $n\to +\infty$, we  deduce
that
\begin{equation}\label{eq201}
\begin{array}{ll}
&\dis\int_{\R^3}(U_{\va_n,r_n}+v_{\va_n,\rho_n})^{2}u^2_n+\dis\int_{\R^3}(U_{\va_n,\rho_n}+v_{\va_n,r_n})^{2}v^2_n
\vspace{0.2cm}\\
=&\dis\int_{\R^3}\Bigl(\sum\limits_{j=1}^\ell
U_{x^j}^2\Bigr)u_n^2+\dis\int_{\R^3}\Bigl(\sum\limits_{j=1}^\ell
U_{y^j}^2\Bigr)v_n^2 +o_n(1).
\end{array}
\end{equation}
Hence we find
\begin{equation}\label{24}
\begin{array}{ll}
o_n(1)&=  \dis\int_{\R^3} \bigl(|\nabla u_n|^2+u_n^2-
 3(U_{\va_n,r_n}+v_{\va_n,\rho_n})^{2}u_n^2\bigr)\vspace{0.2cm}\\
 &\hspace{0.5cm}+ \dis\int_{\R^3} \bigl(|\nabla v_n|^2+v_n^2-
 3(U_{\va_n,\rho_n}+v_{\va_n,r_n})^{2}v_n^2\bigr)+2\va_n\dis\int_{\R^N}u_nv_n \vspace{0.2cm}\\
 &=1+Ce^{-R},
\end{array}
\end{equation}
which is impossible for large $n$ and large $R$.

As a result, we complete the proof.
\end{proof}

\begin{lemma}\label{l1-1-10}
There is a constant $C>0$, independent of $\va$, such that
\[
\| R(\varphi,\psi)\|\le C\|(\varphi,\psi)\|^{3},\quad \|
R'(\varphi,\psi)\|\le C\|(\varphi,\psi)\|^{2},\quad \|
R''(\varphi,\psi)\|\le C\|(\varphi,\psi)\|.
\]
\end{lemma}

\begin{proof} The proof can be completed by direct calculation and
we omit it.
\end{proof}

Now we perform  the finite-dimensional reduction procedure.

\begin{proposition}\label{p1-6-3}
There exists $\varepsilon_0>0$ such that for $\va\in (0,\,\va_0)$,
there is a $C^1$ map from $\mathcal {D}_\va\times \mathcal{D}_\va$
to $H_{s}\times H_{s}$:
$(\varphi,\psi)=(\varphi(r,\rho),\psi(r,\rho))$, satisfying
$(\varphi,\psi)\in \mathbb{E}$, and
\[
 J'_{(\varphi,\psi)}(\varphi,\psi)
=0,\quad
 \hbox{on} \,\,\, \mathbb{E}.
\]
Moreover, there is a constant  $C>0$ independent of $\va$, such that
\begin{equation}\label{2-20-4}
\|(\varphi,\psi)\|\le C\Bigl(
\frac{e^{-|x^1-x^2|}}{|x^1-x^2|}+\frac{e^{-|y^1-y^2|}}{|y^1-y^2|}+\va
e^{-(1-\tau)|x^1-y^1|}+\va^4\Bigr).
\end{equation}
\end{proposition}

\begin{proof} It follows from the proof of Lemma~\ref{l1-25-3} below, that  $l(\varphi,\psi)$ is a
bounded linear functional in $\mathbb{E}$.  Thus,  there is an
$f_\va\in \mathbb{E}$, such that
\[
l(\varphi,\psi)=\bigl\langle f_\va,(\varphi,\psi)\bigr\rangle.
\]
Thus, finding a critical point for $J(\varphi,\psi)$ in $\mathbb{E}$
is equivalent to solving
\begin{equation}\label{1-20-3}
f_\va- L(\varphi,\psi) +R'(\varphi,\psi)=0.
\end{equation}
By Lemma~\ref{2.2}, $L$ is invertible.  Thus, \eqref{1-20-3} can be
rewritten as
\begin{equation}\label{3.20}
(\varphi,\psi)= A(\varphi,\psi)=:L^{-1} (f_\va +R'(\varphi,\psi)).
\end{equation}
Set
\[
\begin{array}{ll}
D=\Bigl\{ (\varphi,\psi):  (\varphi,\psi)\in \mathbb{E},
\|(\varphi,\psi)\|\le
\dis\frac{e^{-(1-\sigma)|x^1-x^2|}}{|x^1-x^2|}&+\dis\frac{e^{-(1-\sigma)|y^1-y^2|}}{|y^1-y^2|}\vspace{0.2cm}\\
&+\va^{1-\sigma} e^{-(1-\tau)|x^1-y^1|}+\va^{4-\sigma} \Bigr\},
\end{array}
\]
where $\sigma>0$ is small.

From Lemma~\ref{l1-1-10} and Lemma~\ref{l1-25-3} below, for $\va$
small,
\begin{equation}\label{3-20-4}
\begin{array}{ll}
\|A(\varphi,\psi)\|&\le C\|f_\va\|+ C\|(\varphi,\psi)\|^{2}\vspace{0.2cm}\\\
 &\le
\dis\frac{e^{-(1-\sigma)|x^1-x^2|}}{|x^1-x^2|}+\dis\frac{e^{-(1-\sigma)|y^1-y^2|}}{|y^1-y^2|}+\va^{1-\sigma}
e^{-(1-\tau)|x^1-y^1|}+\va^{4-\sigma},
 \end{array}
\end{equation}
and
\[
\begin{split}
\|A(\varphi_1,\psi_1)-A(\varphi_2,\psi_2)\| &= \| L^{-1} R'(\varphi_1,\psi_1)-L^{-1} R'(\varphi_2,\psi_2)\|\\
&\le  C\bigl(
\|(\varphi_1,\psi_1)\|+\|(\varphi_2,\psi_2)\|\bigr)\|(\varphi_1,\psi_1)-(\varphi_2,\psi_2)\|\\
&\le \frac12 \|(\varphi_1,\psi_1)-(\varphi_2,\psi_2)\|.
\end{split}
\]
Therefore,     $A$ maps $D$ into $D$ and is a contraction map. So,
 there exists $(\varphi,\psi)\in
\mathbb{E}$, such that  $(\varphi,\psi)=A(\varphi,\psi)$. Moreover
by \eqref{3.20}, we have
\[
\|(\varphi,\psi)\|\le
C\Bigl(\frac{e^{-|x^1-x^2|}}{|x^1-x^2|}+\frac{e^{-|y^1-y^2|}}{|y^1-y^2|}+\va
e^{-(1-\tau)|x^1-y^1|}+\va^4\Bigl).
\]
\end{proof}

\begin{lemma}\label{l1-25-3}
There is a constant  $C>0$ independent of $\va$, such that
\[
\|f_\va\|\le
C\Bigl(\frac{e^{-|x^1-x^2|}}{|x^1-x^2|}+\frac{e^{-|y^1-y^2|}}{|y^1-y^2|}+\va
e^{-(1-\tau)|x^1-y^1|}+\va^4\Bigl).
\]
\end{lemma}

\begin{proof}
We see
\begin{eqnarray*}
&&\dis\int_{\R^3}
\bigl((U_{\va,r}+v_{\va,\rho})^3-\sum\limits_{j=1}^\ell
U_{\va,x^j}^3 -\sum\limits_{j=1}^\ell v_{\va,y^j}^3\bigr)\varphi\\
&=&\int_{\R^3}\Bigl((\sum\limits_{j=1}^\ell
U_{\va,x^j})^3-\sum\limits_{j=1}^\ell
U_{\va,x^j}^3+(\sum\limits_{j=1}^\ell
v_{\va,y^j})^3-\sum\limits_{j=1}^\ell
v_{\va,y^j}^3+3U_{\va,r}^2v_{\va,\rho}+3U_{\va,r}v^2_{\va,\rho}\Bigr)\varphi\\
&=&\int_{\R^3}\Bigl(3\sum\limits_{j \neq i}^\ell
U^2_{\va,x^i}U_{\va,x^j} +3\sum\limits_{j \neq i}^\ell
v_{\va,y^i}^2v_{\va,y^j}+3U_{\va,r}^2v_{\va,\rho}+3U_{\va,r}v^2_{\va,\rho}\Bigr)\varphi.
\end{eqnarray*}

It follows from Proposition~\ref{pro2.1}, Proposition~\ref{pro11}
and H\"{o}lder inequality that for $i \neq j$
\begin{eqnarray*}
&&\Bigl|\int_{\R^3} U^2_{\va,x^i}U_{\va,x^j}\varphi\Bigl|\\
&=&\Bigl|\int_{\R^3}(U_{x^i}+\va^2
p_{\va}(|x-x^i|)+w(x-x^i))^2(U_{x^j}+\va^2
p_{\va}(|x-x^j|)+w(x-x^j))\varphi\Bigl|\\
&\le&C\Bigl(\frac{e^{-|x^i-x^j|}}{|x^i-x^j|}+\va
e^{-(1-\tau)|x^i-x^j|}+\va^4\Bigl)\|\varphi\|_{H^1(\R^3)}, \\
&&\Bigl|\int_{\R^3} v^2_{\va,y^i}v_{\va,y^j}\varphi\Bigl|\\
&=&C\Bigl|\int_{\R^3}(\va^2q_{\va}^2|x-y^i|+h^2(|x-y^i|))(\va
q_{\va}|x-y^j|+h(|x-y^j|))\varphi\Bigl|\\
&\le& C(\va^3 e^{-(1-3\tau)|y^i-y^j|}+\va^4)\|\varphi\|_{H^1(\R^3)},
\end{eqnarray*}
and
$$
\Bigl|\int_{\R^3}(3U_{\va,r}^2v_{\va,\rho}+3U_{\va,r}v^2_{\va,\rho})\varphi\Bigl|\le
C\sum\limits_{i,j=1}^\ell (\va
e^{-(1-\tau)|x^i-y^j|}+\va^4)\|\varphi\|_{H^1(\R^3)}.
$$
Therefore,
 \begin{eqnarray*} &&\Bigl|\dis\int_{\R^3}
\bigl((U_{\va,r}+v_{\va,\rho})^3-\sum\limits_{j=1}^\ell
U_{\va,x^j}^3 -\sum\limits_{j=1}^\ell
v_{\va,y^j}^3\bigr)\varphi\Bigl|\\
&\le&C\Bigl(\sum\limits_{i\neq
j}^\ell\frac{e^{-|x^i-x^j|}}{|x^i-x^j|}+\va\sum\limits_{i, j=1}^\ell
e^{-(1-\tau)|x^i-y^j|}+\va^4\Bigl)\|\varphi\|_{H^1(\R^3)}.
\end{eqnarray*}
Similarly,
\begin{eqnarray*} &&\Bigl|\dis\int_{\R^3}
\bigl((U_{\va,\rho}+v_{\va,r})^3-\sum\limits_{j=1}^\ell
U_{\va,y^j}^3 -\sum\limits_{j=1}^\ell v_{\va,x^j}^3\bigr)\psi\Bigl|\\
&\le&C\Bigl(\sum\limits_{i\neq
j}^\ell\frac{e^{-|y^i-y^j|}}{|y^i-y^j|}+\va\sum\limits_{i, j=1}^\ell
e^{-(1-\tau)|x^i-y^j|}+\va^4\Bigl)\|\psi\|_{H^1(\R^3)}.
\end{eqnarray*}

As a result, we complete the proof.
\end{proof}

Now we are ready to prove Proposition~\ref{pro2}. Let
$(\varphi_{r,\rho},\psi_{r,\rho})=(\varphi(r,\rho),\psi(r,\rho))$ be
the map obtained in Proposition~\ref{p1-6-3}. Define
\[
F(r,\rho)=I(U_{\va,r}+v_{\va,\rho}+\varphi_{r,\rho},
U_{\va,\rho}+v_{\va,r}+\psi_{r,\rho}),\quad \forall\; (r,\rho)\in
\mathcal {D}_\va\times\mathcal {D}_\va.
\]
With the same argument in \cite{Cao,R}, we can easily check that for
$\va$ sufficiently small, if $(r,\rho)$ is a critical point of
$F(r,\rho)$, then $(U_{\va,r}+v_{\va,\rho}+\varphi_{r,\rho},
U_{\va,\rho}+v_{\va,r}+\psi_{r,\rho})$ is a critical point of $I$.

\begin{proof}[Proof of Proposition~\ref{pro2}]

 The boundedness of $L$ in $H_s\times H_s$ and Lemma~ \ref{l1-1-10} imply that
\[
\|L(\varphi_{r,\rho},\psi_{r,\rho})\|\le
C\|(\varphi_{r,\rho},\psi_{r,\rho})\|,\quad
|R(\varphi_{r,\rho},\psi_{r,\rho})|\le
C\|(\varphi_{r,\rho},\psi_{r,\rho})\|^{3}.
\]
So, Proposition~\ref{p1-6-3} and Lemma~\ref{l1-25-3} combined by
Proposition~ \ref{3-1} give
\[
 \begin{split}
F(r,\rho)=&I(U_{\va,r}+v_{\va,\rho},U_{\va,\rho}+v_{\va,r})-l(\varphi_{r,\rho},\psi_{r,\rho})+
\frac12 \bigl\langle
L(\varphi_{r,\rho},\psi_{r,\rho}),(\varphi_{r,\rho},\psi_{r,\rho})\bigr\rangle
 -R(\varphi_{r,\rho},\psi_{r,\rho})\\
=&I(U_{\va,r}+v_{\va,\rho},U_{\va,\rho}+v_{\va,r})+O\bigl(\|f_\va\|\|(\varphi_{r,\rho},\psi_{r,\rho})\|
+\|(\varphi_{r,\rho},\psi_{r,\rho})\|^2\bigr)\\
=&\dis\sum\limits_{j=1}^\ell I(U_{\va,x^j},
v_{\va,x^j})+\dis\sum\limits_{j=1}^\ell I(U_{\va,y^j},
v_{\va,y^j})\\
&-\dis\sum\limits_{i<j}^\ell C_{ij}\frac{e^{-|x^i-x^j|}}{|x^i-x^j|}
-\dis\sum\limits_{i<j}^\ell C_{ij} \frac{e^{-|y^i-y^j|}}{|y^i-y^j|}+ \va\dis\sum\limits_{i,j=1}^\ell \bar C_{ij} e^{-|x^i-y^j|}\\
&+O(\va e^{-(1-\tau)|y^1-y^2|}+\va^2 e^{-(1-\tau)|x^1-y^1|}+\va
e^{-(1-\tau)|x^1-x^2|}+\va^4)\\
&+O\Bigl(\frac{e^{-|x^1-x^2|}}{|x^1-x^2|}+\frac{e^{-|y^1-y^2|}}{|y^1-y^2|}+\va
e^{-(1-\tau)|x^1-y^1|}\Bigl)^2,
\end{split}
\]
where  $\bar C_{ij}$ and $C_{ij}$ are those in Proposition~\ref{3-1}.

Recalling
$$
r,\rho\in \mathcal{D}_\va =:  \Bigl[ \frac
{|\ln\va|}{m-n+\frac{\mu\ln|\ln\va|}{|\ln\va|}},\,\frac
{|\ln\va|}{m-n} \Bigr],
$$
where $m=2\sin\frac{\pi}{\ell},\, n=\sqrt{2(1-\cos
\frac{\pi}{\ell})}$, $\mu>m-n>0$, and noting
\begin{equation}\label{345}
\frac
{1}{m-n+\frac{\mu\ln|\ln\va|}{|\ln\va|}}=\frac{1}{m-n}-\frac{\mu}{(m-n)^2}\frac{\ln|\ln\va|}{|\ln\va|}
+O\Bigl(\frac{\ln|\ln\va|}{|\ln\va|}\Bigl)^2,
\end{equation}
we can check
\begin{eqnarray*}
&&O(\va e^{-(1-\tau)|y^1-y^2|}+\va^2 e^{-(1-\tau)|x^1-y^1|}+\va
e^{-(1-\tau)|x^1-x^2|}+\va^4)\\
&&+O\Bigl(\frac{e^{-|x^1-x^2|}}{|x^1-x^2|}+\frac{e^{-|y^1-y^2|}}{|y^1-y^2|}+\va
e^{-(1-\tau)|x^1-y^1|}\Bigl)^2\\
&=& O(\va^{\frac{m}{m-n}+\sigma}),
\end{eqnarray*}
where $\sigma>0$ is a small number such that  $2<\frac{m}{m-n}+\sigma<4$
for $\ell\ge 2$.

Hence, considering the symmetry again, we find
\begin{equation}\label{348}
 \begin{split}
F(r,\rho)=&\dis\sum\limits_{j=1}^\ell I(U_{\va,x^j},
v_{\va,x^j})+\dis\sum\limits_{j=1}^\ell I(U_{\va,y^j},
v_{\va,y^j})\\
&-\dis\sum\limits_{i<j}^\ell C_{ij} \frac{e^{-|x^i-x^j|}}{|x^i-x^j|}
-\dis\sum\limits_{i<j}^\ell C_{ij}\frac{e^{-|y^i-y^j|}}{|y^i-y^j|}+
\va\dis\sum\limits_{i,j=1}^\ell \bar C_{ij}
e^{-|x^i-y^j|}+O(\va^{\frac{m}{m-n}+\sigma})\\
=&C_\va+\ell\Bigl(\bar C\va e^{-\sqrt{r^2+\rho^2-2r\rho
\cos\frac{\pi}{\ell}}}-\frac{C}{mr}e^{-mr}-\frac{C}{m\rho}e^{-m\rho}\Bigl)+O(\va^{\frac{m}{m-n}+\sigma}),
\end{split}
\end{equation}
where $C_\va$ depends on $\va$ but is independent of $r$ and $\rho$, $C=C_{12},\,\bar C=\bar C_{11}$.

Now we prove that the maximizer of $F(r,\rho)$ in
$\mathcal{D}_\va\times \mathcal{D}_\va$ is  an interior point of
$\mathcal{D}_\va\times \mathcal{D}_\va$. To this end, we consider
the function
$$
G(r,\rho)=\bar C\va e^{-\sqrt{r^2+\rho^2-2r\rho
\cos\frac{\pi}{\ell}}}-\frac{C}{mr}e^{-mr}-\frac{C}{m\rho}e^{-m\rho},\,\,\,r,\rho\in
\mathcal{D}_\va.
$$
In order to check that $G(r,\rho)$ achieves maximum at some point
$(r_0,\rho_0)$ in the interior of $\mathcal{D}_\va\times
\mathcal{D}_\va$, we need to estimate both the value of $G(r,\rho)$ on
the boundary of $\mathcal{D}_\va\times \mathcal{D}_\va$ and the value of
 $G(r_0,\rho_0)$.

Set
$$
\check{r}=\frac
{|\ln\va|}{m-n+\frac{\mu\ln|\ln\va|}{|\ln\va|}},\,\hat{r}=\frac
{|\ln\va|}{m-n},
$$
and define
$$
\rho_\theta=\frac
{|\ln\va|}{m-n+\frac{\mu\ln|\ln\va|}{|\ln\va|}\theta},\quad\,\theta\in
[0,1].
$$
Then $\check{r}\le \rho_\theta\le\hat{r}$ for $\theta\in [0,1]$, and
\begin{equation}\label{349}
 \rho_\theta=\frac{|\ln\va|}{m-n}-\frac{\mu\theta}{(m-n)^2}\ln|\ln\va|
+O\Bigl(\frac{\ln^2|\ln\va|}{|\ln\va|}\Bigl),
\end{equation}
\begin{equation}\label{460}
\sqrt{\hat{r}^2+\rho_\theta^2-2\hat{r}\rho_\theta\cos\frac{\pi}{\ell}}=
\frac{n}{m-n}|\ln\va|-\frac{\mu\theta
n}{2(m-n)^2}\ln|\ln\va|+O\Bigl(\frac{\ln^2|\ln\va|}{|\ln\va|}\Bigr).
\end{equation}
Hence
\begin{equation}\label{467}
\begin{array}{ll}
G(\hat{r},\rho_\theta)=&-C\va^{\frac{m}{m-n}}|\ln\va|^{-1}-c\va^{\frac{m}{m-n}}|\ln\va|^{\frac{\mu\theta
m}{(m-n)^2}-1}+\tilde c\va^{\frac{m}{m-n}}|\ln\va|^{\frac{\mu\theta
n}{2(m-n)^2}}\vspace{0.2cm}\\
&+o\Bigl(\va^{\frac{m}{m-n}}|\ln\va|^{\frac{\mu\theta
n}{2(m-n)^2}}+\va^{\frac{m}{m-n}}|\ln\va|^{\frac{\mu\theta
m}{(m-n)^2}-1}\Bigr),
\end{array}
\end{equation}
where $C,\,c$ and $\tilde c$ are positive constants independent of
$\va$.

Set
$$
f(\theta)=\frac{\mu\theta m}{(m-n)^2}-1-\frac{\mu\theta n}{2(m-n)^2}.
$$
Since $\mu>m-n>0$, we see
$$
f(1)=\frac{\mu m}{(m-n)^2}-1-\frac{\mu n}{2(m-n)^2}> \frac{\mu m}{(m-n)^2}-1-\frac{\mu n}{(m-n)^2}=\frac{\mu}{m-n}-1>0.
$$
Considering $f(0)=-1<0$, there exists  a unique  $\bar\theta\in (0,1)$ such that
\begin{equation}\label{469}
\frac{\mu\bar\theta m}{(m-n)^2}-1=\frac{\mu\bar\theta n}{2(m-n)^2}.
\end{equation}
Moreover, if $\theta\in (\bar\theta,\,1]$, then
$$
\frac{\mu\theta m}{(m-n)^2}-1>\frac{\mu\theta n}{2(m-n)^2},
$$
which implies $G(\hat r,\rho_\theta)<0$. But, if $\theta\in [0,\,\bar\theta)$, then
\begin{equation}\label{4691}
\frac{\mu\theta m}{(m-n)^2}-1<\frac{\mu\theta n}{2(m-n)^2},
\end{equation}
which means  $G(\hat r,\rho_\theta)>0$ and
$$
 G(\hat r,\rho_\theta)< c_1
\va^{\frac{m}{m-n}}|\ln\va|^{\frac{\mu\bar\theta
n}{2(m-n)^2}}
$$
for some constant $c_1>0$ independent of $\va$.

Therefore, we get that for $\va$ sufficiently small
$$
G(\hat{r},\rho_\theta)\le 2c_1
\va^{\frac{m}{m-n}}|\ln\va|^{\frac{\mu\bar\theta
n}{2(m-n)^2}},\quad\,\,\forall\,\, \theta\in [0,1],
$$
which says that
\begin{equation}\label{470}
\max\limits_{\rho\in \mathcal{D}_\va} G(\hat{r},\rho)\le 2c_1
\va^{\frac{m}{m-n}}|\ln\va|^{\frac{\mu\bar\theta n}{2(m-n)^2}}.
\end{equation}
Similarly,
\begin{equation}\label{480}
\max\limits_{r\in \mathcal{D}_\va} G(r,\hat{r})\le 2c_1
\va^{\frac{m}{m-n}}|\ln\va|^{\frac{\mu\bar\theta n}{2(m-n)^2}}.
\end{equation}

\begin{remark}\label{re3}
It can be verified from \eqref{467} and \eqref{4691} that for
$\theta\in (0,\,\bar\theta)$, there exists a constant  $c_2>0$
independent of $\va$ such that for $\va$ sufficiently small,
$$
\max\limits_{\rho\in \mathcal{D}_\va} G(\hat r,\rho)\ge c_2
\va^{\frac{m}{m-n}}|\ln\va|^{\frac{\mu\theta n}{2(m-n)^2}}.
$$

\end{remark}

Now we estimate $\max\limits_{\rho\in \mathcal{D}_\va}
G(\check{r},\rho)$.

Since for $\va>0$ sufficiently small,
$$
\sqrt{\check{r}^2+\rho^2-2\check{r}\rho \cos\frac{\pi}{\ell}}\ge
n\check{r},
$$
it follows from \eqref{345} and the fact $\mu>m-n>0$ that for $\va$
sufficiently small,
\begin{eqnarray*}
G(\check{r},\rho)&\le&-\frac{C}{m\check{r}}e^{-m\check{r}}+\bar C\va
e^{-\sqrt{\check{r}^2+\rho^2-2\check{r}\rho
\cos\frac{\pi}{\ell}}}\\
&\le & -\frac{C}{m\check{r}}e^{-m\check{r}}+\bar C\va e^{-n\check{r}}\\
&\le&-C_1\va^{\frac{m}{m-n}}|\ln\va|^{\frac{\mu
m}{(m-n)^2}-1}+\bar C_1\va^{\frac{m}{m-n}}|\ln\va|^{\frac{\mu
n}{(m-n)^2}}\\
&<& 0,
\end{eqnarray*}
where $ C_1$ and $\bar C_1$ are positive constants independent of
$\va$. Hence,
\begin{equation}\label{490}
\max\limits_{\rho\in \mathcal{D}_\va} G(\check{r},\rho)\le 0.
\end{equation}

The same argument yields
\begin{equation}\label{495}
\max\limits_{r\in \mathcal{D}_\va} G(r,\check{r})\le 0.
\end{equation}

At last, we  estimate $G(r_0,\rho_0))$. Taking $\theta=\bar\theta$
in \eqref{349}, we find for $\va$ sufficiently small
\begin{eqnarray*}
G(r_0,\rho_0)&\ge& G(\rho_{\bar\theta},\rho_{\bar\theta})\\
&=&\bar C\va e^{-n\rho_{\bar\theta}}-\frac{2C}{m\rho_{\bar\theta}}e^{-m\rho_{\bar\theta}}\\
&=& \bar
C\va^{\frac{m}{m-n}}|\ln\va|^{\frac{\mu\bar\theta n}{(m-n)^2}} -\frac{2(m-n)C}{m}\va^{\frac{m}{m-n}}|\ln\va|^{\frac{\mu\bar\theta
m}{(m-n)^2}-1}+o\Bigl(\va^{\frac{m}{m-n}}|\ln\va|^{\frac{\mu\bar\theta n}{(m-n)^2}} \Bigl)\\
&\ge& \frac{\bar C}{2}\va^{\frac{m}{m-n}}|\ln\va|^{\frac{\mu\bar\theta
n}{(m-n)^2}},
\end{eqnarray*}
since by \eqref{469},
$
\frac{\mu\bar\theta n}{(m-n)^2}>\frac{\mu\bar\theta
n}{2(m-n)^2}=\frac{\mu\bar\theta m}{(m-n)^2}-1.
$

 The above estimate and \eqref{470}-\eqref{495} show that for $\va>0$ sufficiently small, $(r_0,\rho_0)$ is
  in the interior of
$\mathcal{D}_\va\times \mathcal{D}_\va$. Comparing  the above estimate on $G(r_0,\rho_0)$ and \eqref{470}-\eqref{495} with \eqref{348}, we
conclude that $F(r,\rho)$  achieves
 (local) maximum also in the interior of
$\mathcal{D}_\va\times \mathcal{D}_\va$.

As a consequence, we complete the proof.
\end{proof}

\section{ Segregated solutions for system coupled by three equations}

In this section, we consider the following  system linearly coupled
by three equations
\begin{equation}\label{eqmain4}
\left\{
\begin{array}{ll}
-\Delta u+u=u^3-\varepsilon (v+\omega),\hspace{1cm}& x\in \R^3, \vspace{0.2cm}\\
-\Delta v+v=v^3 -\varepsilon (u+\omega),& x\in\R^3,\vspace{0.2cm}\\
-\Delta \omega+\omega=\omega^3 -\varepsilon (u+v),& x\in\R^3.
\end{array}
\right.
\end{equation}

Let $(U_\va,v_\va,\omega_\va)\in (H^1_r(\R^3))^3$ be the solution of \eqref{eqmain4}
obtained in Proposition~\ref{pro2.2}. In this part, we will use the same
notations  as those in previous sections.  Define
$$
\omega_{\va,r}=\sum\limits_{j=1}^\ell
\omega_{\va,x^j},\,\,\,\omega_{\va,\rho}=\sum\limits_{j=1}^\ell
\omega_{\va,y^j}
$$
and
$$
\mathbf{E}=\{(\varphi,\psi,\phi)\in
(H^1_r(\R^3))^3:\,\,(\varphi,\psi)\in \mathbb{E},\,\phi\in H_s\},
$$
where $\mathbb{E}$ is defined as \eqref{E},  $r,\,\rho\in
\mathcal{D}_\va$ and $\mathcal{\mathcal{D}_\va}$ is defined by
\eqref{1-20-5} for $\ell>2$ but by
$$
\mathcal{D}_{\va}=\Bigl[\frac
{|\ln\va|}{1+\frac{\mu\ln|\ln\va|}{|\ln\va|}},\,|\ln\va|\Bigr],\,\,(\mu>1)
$$
for $\ell=2$.

To prove Theorem~\ref{main2}, we only need to verify
\begin{proposition}\label{pro41}
For any integer $\ell\ge 2$, there exists $\varepsilon_0>0$ such
that for $\va\in (0,\va_0)$, problem~\ref{eqmain4} has a solution
$(u,v,\omega)$ with the form
$$
\left\{
\begin{array}{ll}
u=U_{\va,r}+v_{\va,\rho}+\omega_\va+\varphi,\vspace{0.2cm}\\
v=\omega_{\va,r}+U_{\va,\rho}+v_\va+\psi,\vspace{0.2cm}\\
\omega=v_{\va,r}+\omega_{\va,\rho}+U_\va+\phi,
\end{array}
\right.
$$
where $(\varphi,\psi,\phi)\in \mathbf{E}$ satisfies
$$
\|(\varphi,\psi,\phi)\|= \left\{
\begin{array}{ll}
o(\va^{\frac{m}{m-n}}),\quad
&\hbox{if}\,\,\,\ell>2,\vspace{0.2cm}\\
o(\va^2),&\hbox{if}\,\,\,\ell=2.
\end{array}
\right.
$$
\end{proposition}
\begin{proof}
The proof is similar to that of Proposition~\ref{pro2}, we only give
a sketch here.

Define
\[
\begin{array}{ll}
\bar I(u,v,\omega)=&\dis\frac12\int_{\R^3} \bigl( |\nabla u|^2+ u^2+
|\nabla v|^2+
v^2+|\nabla \omega|^2+ \omega^2\bigl)\vspace{0.2cm}\\
&-\dis\frac1{4}\int_{\R^3}\bigl(u^{4}+v^{4}+\omega^4\bigl)+\va\int_{\R^3}(uv+u\omega+v\omega),\quad
\forall\,\,\,(u,v,\omega)\in (H_s)^3,
\end{array}
\]
and
\[
\begin{array}{ll}
\bar J(\varphi,\psi,\phi)=& \bar
I(U_{\va,r}+v_{\va,\rho}+\omega_\va+\varphi,\,\omega_{\va,r}+U_{\va,\rho}+v_\va+\psi,\,
v_{\va,r}+\omega_{\va,\rho}+U_\va+\phi),\vspace{0.2cm}\\
&\quad\quad \forall\,\,\,(\varphi,\psi,\phi)\in \mathbf{E}.
\end{array}
\]

Proceeding as we prove Proposition~\ref{p1-6-3}, we find that for
$\va$ sufficiently small, there is a $C^1$ map from $(\mathcal
{D}_\va)^2$ to $\mathbf{E}$:
$(\varphi,\psi,\phi)=(\varphi(r,\rho),\psi(r,\rho),\phi(r,\rho))$,
satisfying
\[
 \bar J'_{(\varphi,\psi,\phi)}(\varphi,\psi,\phi)
=0,\quad
 \hbox{on} \,\,\, \mathbf{E},
\]
and
\begin{equation}\label{2-20-84}
\|(\varphi,\psi,\phi)\|= O\Bigl(
\frac{e^{-|x^1-x^2|}}{|x^1-x^2|}+\frac{e^{-|y^1-y^2|}}{|y^1-y^2|}+\va
e^{-(1-\tau)|x^1-y^1|}+\va e^{-(1-\tau)|x^1|}+\va
e^{-(1-\tau)|y^1|}+\va^4\Bigr).
\end{equation}
We should point out here that  when we carry out  the finite
dimensional reduction, we do not impose an orthogonal decomposition
on $\phi$ (see the definition of $\mathbf{E}$), since the kernel of
the operator $\Delta-(1-3U^2)I$ in $H_s$ is $\{0\}$.

It follows from Proposition~\ref{491} and \eqref{2-20-84} that
\[
 \begin{split}
&\bar F(r,\rho)=:\bar J(\varphi(r,\rho),\psi(r,\rho),\phi(r,\rho))\\
=&\dis\sum\limits_{j=1}^\ell \bar I(U_{\va,x^j}, v_{\va,x^j},
\omega_{\va,x^j})+\dis\sum\limits_{j=1}^\ell \bar I(U_{\va,y^j},
v_{\va,y^j},\omega_{\va,y^j})+\bar I(U_\va,v_\va,\omega_\va)\\
&-\dis\sum\limits_{i<j}^\ell C_{ij}\frac{e^{-|x^i-x^j|}}{|x^i-x^j|}
-\dis\sum\limits_{i<j}^\ell C_{ij} \frac{e^{-|y^i-y^j|}}{|y^i-y^j|}+ \va\dis\sum\limits_{i,j=1}^\ell \bar C_{ij} e^{-|x^i-y^j|}
+\va\dis\sum\limits_{j=1}^\ell \tilde C_j (e^{-|x^j|}+e^{-|y^j|})\\
&+O(\va e^{-(1-\tau)|y^1-y^2|}+\va
e^{-(1-\tau)|x^1-x^2|}+\va^2 (e^{-(1-\tau)|x^1|}+e^{-(1-\tau)|y^1|}+e^{-(1-\tau)|x^1-y^1|})+\va^4)\\
&+O\Bigl(\frac{e^{-|x^1-x^2|}}{|x^1-x^2|}+\frac{e^{-|y^1-y^2|}}{|y^1-y^2|}+\va
e^{-(1-\tau)|x^1-y^1|}+\va e^{-(1-\tau)|x^1|}+\va
e^{-(1-\tau)|y^1|}\Bigl)^2\\
=&C_\va+\ell\Bigl(\bar C\va e^{-\sqrt{r^2+\rho^2-2r\rho
\cos\frac{\pi}{\ell}}}+\tilde C \va (e^{-r}+e^{-\rho})
-\frac{C}{mr}e^{-mr}-\frac{C}{m\rho}e^{-m\rho}\Bigl)+O(\va^{\frac{m}{m-n}+\sigma}),
\end{split}
\]
where $C_\va>0$ depends on $\va$ but is independent of $r$ and $\rho$. $\bar C,\,\tilde C$ and $C$ are positive constants independent of $\va,\,r$ and $\rho$.

Define function
$$
\bar G(r,\rho)=\bar C\va e^{-\sqrt{r^2+\rho^2-2r\rho
\cos\frac{\pi}{\ell}}}+\tilde C \va
(e^{-r}+e^{-\rho})-\frac{C}{mr}e^{-mr}-\frac{C}{m\rho}e^{-m\rho},\,\,\,r,\rho\in
\mathcal{D}_\va.
$$
We want  to verify that $\bar G(r,\rho)$ achieves maximum at some
point $(r_0,\rho_0)$ which is in the interior of
$\mathcal{D}_\va\times \mathcal{D}_\va$. We have three cases:
(1)\,\,$\ell=2$; (2)\,\,$\ell=3$;\,\,(3)\,\,$\ell>3$.

{\bf Case (1)}:\,\,$\ell=2$.

In this situation, $m=2,\,n=\sqrt{2}$,
$|x^1-y^1|=\sqrt{r^2+\rho^2-2r\rho
\cos\frac{\pi}{\ell}}>(1+\sigma)\max\{r,\,\rho\}$ for some
$\sigma>0$. Hence, without loss of generality, we suppose
\begin{eqnarray*}
\bar G(r,\rho)&=&(\tilde C \va e^{-r}-\frac{C}{2r}e^{-2r})+(\tilde
C \va e^{-\rho}-\frac{C}{2\rho}e^{-2\rho})\\
&=:& G(r)+G(\rho),\,\,\,r,\rho\in \mathcal{D}_\va.
\end{eqnarray*}
Therefore we need to modify $\mathcal{D}_\va$ as
$$
\mathcal{D}_{\va}=\Bigl[\frac
{|\ln\va|}{1+\frac{\mu\ln|\ln\va|}{|\ln\va|}},\,|\ln\va|\Bigr],\quad
\mu>1.
$$

 Using the argument to prove Proposition~\ref{pro2} (see Remark~\ref{re3}), we can find
$\bar r_0$  which is interior points of $\mathcal{D}_\va$ such that
\begin{eqnarray*}
&G(\bar r_0)=\max\limits_{r\in \mathcal{D}_\va}{G(r)}\ge
C_1\va^2|\ln\va|^{\tilde\theta}\ge C_1\va^2\ge
C_1\va^{\frac{2}{2-\sqrt{2}}}
\end{eqnarray*}
for some $\tilde\theta>0$ and $C_1>0$.

 Suppose that $\bar G(r,\rho)$ achieves
maximum at $(r_0,\rho_0)\in \mathcal{D}_\va\times \mathcal{D}_\va$,
then
\begin{equation}\label{420}
\bar G(r_0,\rho_0)\ge 2G(\bar r_0)\ge
2C_1\va^2|\ln\va|^{\tilde\theta}.
\end{equation}

On the other hand, there exist $\sigma>0$ and $C_2>0$ such that
\begin{equation}\label{425}
\begin{array}{ll}
&\bar G(\check{r},\rho)\le -C_2\va^2|\ln\va|^{2\mu-1}+G(\bar
r_0)<(1-\sigma)\bar G(r_0,\rho_0),\quad\,\,\forall\,\,\rho\in \mathcal{D}_\va,\vspace{0.2cm}\\
&\bar G(\hat{r},\rho)\le C_2\va^2+G(\bar r_0)<(1-\sigma)\bar
G(r_0,\rho_0),\quad\,\,\forall\,\,\rho\in \mathcal{D}_\va,\vspace{0.2cm}\\
&\bar G(r,\check r)\le G(\bar
r_0)-C_2\va^2|\ln\va|^{2\mu-1}<(1-\sigma)\bar G(r_0,\rho_0),\quad\,\,\forall\,\,r\in \mathcal{D}_\va,\vspace{0.2cm}\\
&\bar G(r,\hat r)\le G(\bar r_0)+C_2\va^2<(1-\sigma)\bar
G(r_0,\rho_0),\quad\,\,\forall\,\,r\in \mathcal{D}_\va,
\end{array}
\end{equation}
where $\check r$ and $\hat r$ are modified respectively  as
$$
\check r=\frac {|\ln\va|}{1+\frac{\mu\ln|\ln\va|}{|\ln\va|}},\quad
\hat r=|\ln\va|.
$$

Therefore, $(r_0,\rho_0)$ is an interior point of
$\mathcal{D}_\va\times \mathcal{D}_\va$. Comparing \eqref{425}
with $\bar F(r,\rho)$ and  \eqref{420}, we conclude that $\bar
F(r,\rho)$ has (local) maximizer in the interior of
$\mathcal{D}_\va\times \mathcal{D}_\va$.

{\bf Case (2)}:\,\,$\ell=3$.

In this case, $m=\sqrt{3},\,n=1$, and it is possible that $r\sim
\rho \sim |x^1-y^1|=\sqrt{r^2+\rho^2-2r\rho \cos\frac{\pi}{\ell}}$.
So we consider
$$
\bar G(r,\rho)=\bar C\va e^{-\sqrt{r^2+\rho^2-r\rho }}+\tilde C \va
(e^{-r}+e^{-\rho})-\frac{C}{\sqrt{3}r}e^{-\sqrt{3}r}-\frac{C}{\sqrt{3}\rho}e^{-\sqrt{3}\rho},\,\,\,r,\rho\in
\mathcal{D}_\va.
$$

Now we analyze  $\bar G(r,\rho)$ on $\partial(\mathcal{D}_\va\times
\mathcal{D}_\va)$.

Firstly, again using \eqref{345}, we see
\begin{equation}\label{432}
\begin{array}{ll}
&\bar G(\check{r},\rho)\le \hat C \va
e^{-\check{r}}-\dis\frac{C}{\sqrt{3}\check{r}}e^{-\sqrt{3}\check{r}}<0,\forall\,\,\rho\in
\mathcal{D}_\va,\vspace{0.2cm}\\
&\bar G(r,\check{r})\le \hat C \va
e^{-\check{r}}-\dis\frac{C}{\sqrt{3}\check{r}}e^{-\sqrt{3}\check{r}}<0,\forall\,\,r\in
\mathcal{D}_\va,
\end{array}
\end{equation}
where $\hat C$ and $C$ are positive constants independent of $\va$.

On the other hand, suppose that, in $\mathcal{D}_\va$, $\bar
G(\hat{r},\rho)$ achieves maximizer $\bar\rho\in (\check r,\,\hat
r)$. Arguing as we prove Proposition~\ref{pro2} (see
Remark~\ref{re3}), we find
\begin{equation}\label{450}
\bar G(\hat{r},\bar\rho)\ge
C_3\va^{\frac{\sqrt{3}}{\sqrt{3}-1}}|\ln\va|^{\theta_0}
\end{equation}
for some $\theta_0>0$ and $C_3>0$.

 Since
$$
\tilde C \va
e^{-\hat{r}}-\frac{C}{\sqrt{3}\hat{r}}e^{-\sqrt{3}\hat{r}}=O(\va^{\frac{\sqrt{3}}{\sqrt{3}-1}})
$$
and
$$
e^{-\sqrt{\hat{r}^2+\bar\rho^2-\hat{r}\bar\rho}}< e^{-\bar\rho},
$$
we see
\begin{eqnarray*}
C_3\va^{\frac{\sqrt{3}}{\sqrt{3}-1}}|\ln\va|^{\theta_0}&\le& \bar
G(\hat{r},\bar\rho)<\bar C\va
e^{-\bar\rho}+O(\va^{\frac{\sqrt{3}}{\sqrt{3}-1}})+\tilde C\va
e^{-\bar\rho}-\frac{C}{\sqrt{3}\bar\rho}e^{-\sqrt{3}\bar\rho}.
\end{eqnarray*}
Hence, there exists $\sigma>0$ such that
\begin{eqnarray*}
\bar G(\bar\rho,\bar\rho)&=&\bar C\va e^{-\bar\rho}+2\bigl(\tilde
C\va
e^{-\bar\rho}-\frac{C}{\sqrt{3}\bar\rho}e^{-\sqrt{3}\bar\rho}\bigr)\\
&>& \bar G(\hat{r},\bar\rho)+\tilde C\va
e^{-\bar\rho}-\frac{C}{\sqrt{3}\bar\rho}e^{-\sqrt{3}\bar\rho}\\
&>& (1+\sigma)\bar G(\hat{r},\bar\rho),
\end{eqnarray*}
which implies
$$
\max\limits_{r,\rho\in \mathcal{D}_\va}\bar G(r,\rho)\ge \bar
G(\bar\rho,\bar\rho)>(1+\sigma)\bar
G(\hat{r},\bar\rho)=(1+\sigma)\max\limits_{\rho\in
\mathcal{D}_\va}\bar G(\hat{r},\rho).
$$
Similarly,
$$
\max\limits_{r,\rho\in \mathcal{D}_\va}\bar G(r,\rho)\ge
(1+\sigma)\max\limits_{r\in \mathcal{D}_\va}\bar G(r,\hat r).
$$
These two estimates and \eqref{432} imply that   $\bar F(r,\rho)$
has (local) maximizer in the interior of $\mathcal{D}_\va\times
\mathcal{D}_\va$.

{\bf Case (3)}:\,\,$\ell>3$.

In this situation, $\sqrt{r^2+\rho^2-2r\rho
\cos\frac{\pi}{\ell}}=|x^1-y^1|<(1-\sigma) \min\{r,\,\rho\}$ for
some $\sigma>0$. Then
$$
\bar G(r,\rho)=\bar C\va e^{-\sqrt{r^2+\rho^2-2r\rho
\cos\frac{\pi}{\ell}}}-\frac{C}{mr}e^{-mr}-\frac{C}{m\rho}e^{-m\rho},\,\,\,r,\rho\in
\mathcal{D}_\va.
$$

This is exactly $G(r,\rho)$ in the proof of Proposition~\ref{pro2}
and we omit the rest of the proof.

As a result, we complete the proof.
\end{proof}

\appendix

\section{Energy expansions}

In this section, we will expand the energy
$I(U_{\va,r}+v_{\va,\rho},\,v_{\va,r}+U_{\va,\rho})$, which is
defined as
\[
\begin{array}{ll}
I(u,v)=&\dis\frac12\int_{\R^3} \bigl( |\nabla u|^2+ u^2+ |\nabla
v|^2+v^2\bigl)\vspace{0.2cm}\\
&-\dis\frac1{4}\int_{\R^3}\bigl(u^4+v^{4}\bigl)+\va\int_{\R^3}uv,\quad
(u,v)\in H_s\times H_s.
\end{array}
\]

Recall that $(U_\va,v_\va)$ has the form
\begin{equation}\label{eq41}
\begin{array}{ll}
U_\varepsilon= U+ \va^2 p_\va(r)+w, \,\,\,v_\varepsilon= \va
q_\va(r)+h,
\end{array}
\end{equation}
where
\begin{equation}\label{eq42}
p_\va(r)\le Ce^{-(1-\tau)r},\,\,q_\va(r)\le
Ce^{-(1-\tau)r},\,\,\,\|(w,h)\|\le C\va^4.
\end{equation}

The following estimates can be found  in \cite{ACR}.

\begin{proposition}\label{pro11}
Suppose  that $u(x),v(x)\in H_r^1(\R^N)\,(N\ge 1)$ satisfy
$$
u(r)\sim r^\alpha e^{-\beta r},\,\,\,v(r)\sim r^\gamma e^{-\eta
r},\,\,\,(r\to +\infty),
$$
where $\alpha,\,\gamma\in \R,\,\beta>0,\,\eta>0$. Let $y\in\R^N$
with $|y|\to +\infty$. We have

(i)\,\,if $\beta< \eta$, then
$$
\int_{\R^N}u_{y}v\sim |y|^{\alpha}e^{-\beta|y|}.
$$

(ii)\,\, if $\beta=\eta$, suppose for simplicity, that $\alpha\ge
\gamma$. Then
$$
\int_{\R^N}u_{y}v\sim \left\{
\begin{array}{ll}
 e^{-\beta|y|}|y|^{\alpha+\gamma+\frac{1+N}2}\,\,\,\,\,&\hbox{if}\,\,\gamma>-\frac{1+N}2,\vspace{0.2cm}\\
e^{-\beta|y|}|y|^{\alpha}\ln
|y|\,\,\,\,\,&\hbox{if}\,\,\gamma=-\frac{1+N}2,\vspace{0.2cm}\\
e^{-\beta|y|}|y|^{\alpha}\,\,\,\,\,&\hbox{if}\,\,\gamma<-\frac{1+N}2.
\end{array}
\right.
$$

\end{proposition}

\begin{proposition}\label{3-1} We have
\begin{eqnarray*}
&&I(U_{\va,r}+v_{\va,\rho},\,v_{\va,r}+U_{\va,\rho})\\
&=&\dis\sum\limits_{j=1}^\ell I(U_{\va,x^j},
v_{\va,x^j})+\dis\sum\limits_{j=1}^\ell I(U_{\va,y^j},
v_{\va,y^j})\\
&&-\dis\sum\limits_{i<j}^\ell C_{ij}\frac{e^{-|x^i-x^j|}}{|x^i-x^j|}
-\dis\sum\limits_{i<j}^\ell C_{ij}\frac{e^{-|y^i-y^j|}}{|y^i-y^j|}+\va\dis\sum\limits_{i,j=1}^\ell\bar C_{ij} e^{-|x^i-y^j|}\\
&&+O(\va e^{-(1-\tau)|y^1-y^2|}+\va^2 e^{-(1-\tau)|x^1-y^1|}+\va
e^{-(1-\tau)|x^1-x^2|}+\va^4),
\end{eqnarray*}
where $C_{ij},\bar C_{ij}\,(i,j=1,\cdots,\ell)$ are positive constants independent of $\va$, $r$
and $\rho$.
\end{proposition}

\begin{proof}
Write
\begin{equation}\label{43}
\begin{array}{ll}
&I(U_{\va,r}+v_{\va,\rho},\,v_{\va,r}+U_{\va,\rho})\vspace{0.2cm}\\
=&I(U_{\va,r}, v_{\va,r})+I(U_{\va,\rho}, v_{\va,\rho})\vspace{0.2cm}\\
&-\dis\frac1 4 \int_{\R^3}\Bigl(\bigl(
U_{\va,r}+v_{\va,\rho}\bigr)^4-U^4_{\va,r}-v^4_{\va,\rho}-4\dis\sum\limits_{i,j=1}^\ell
U^3_{\va,x^i}v_{\va,y^j} \Bigr)\vspace{0.2cm}\\
&-\dis\frac1 4 \int_{\R^3}\Bigl(\bigl(
U_{\va,\rho}+v_{\va,r}\bigr)^4-U^4_{\va,\rho}-v^4_{\va,r}-4\dis\sum\limits_{i,j=1}^\ell
U^3_{\va,y^i}v_{\va,x^j} \Bigr)\vspace{0.2cm}\\
&+\va\dis\int_{\R^3}\Bigl((U_{\va,r}+v_{\va,\rho})(v_{\va,r}+U_{\va,\rho})-U_{\va,r}v_{\va,r}-
U_{\va,\rho}v_{\va,\rho}-2\dis\sum\limits_{i,j=1}^\ell
v_{\va,x^i}v_{\va,y^j}\Bigr)\vspace{0.2cm}\\
=:&I(U_{\va,r}, v_{\va,r})+I(U_{\va,\rho}, v_{\va,\rho})-I_1-I_2+\va
I_3.
\end{array}
\end{equation}
Now we estimate each term  in \eqref{43}.

For $I_1$, from \eqref{eq42} we see
\begin{equation}\label{44}
\begin{array}{ll}
I_1&=\dis\int_{\R^3}\Bigl[4\Bigl(\sum\limits_{i=1}^\ell
U_{\va,x^i}\Bigr)^3\sum\limits_{i=1}^\ell
v_{\va,y^i}-4\dis\sum\limits_{i,j=1}^\ell
U^3_{\va,x^i}v_{\va,y^j}+4\sum\limits_{i=1}^\ell
U_{\va,x^i}\Bigl(\sum\limits_{i=1}^\ell v_{\va,y^i}\Bigr)^3\Bigr]\vspace{0.2cm}\\
&\hspace{0.4cm}+O\Bigl(\dis\int_{\R^N}\Bigl(\sum\limits_{i=1}^\ell
U_{\va,x^i}\Bigl)^2\Bigl(\sum\limits_{i=1}^\ell
v_{\va,y^i}\Bigr)^2\Bigl)\vspace{0.2cm}\\
&=O\Bigl(\dis\int_{\R^3}\sum\limits_{i\neq j}^\ell
U^2_{\va,x^i}U_{\va,x^j}\sum\limits_{i=1}^\ell
v_{\va,y^i}\Bigl)+O(\va^3e^{-(1-\tau)|x^1-y^1|}+\va^2e^{-(1-\tau)|x^1-y^1|}+\va^4)\vspace{0.2cm}\\
&=O(\va
e^{-(1-\tau)|x^1-x^2|}+\va^4)+O(\va^2e^{-(1-\tau)|x^1-y^1|}).
\end{array}
\end{equation}

Similarly,
\begin{equation}\label{45}
\begin{array}{ll}
I_2=O(\va
e^{-(1-\tau)|y^1-y^2|}+\va^4)+O(\va^2e^{-(1-\tau)|x^1-y^1|}).
\end{array}
\end{equation}

Calculating $I_3$, we obtain
$$
I_3=\int_{\R^3}\Bigl(\dis\sum\limits_{i,j=1}^\ell
U_{\va,x^i}U_{\va,y^j}-\dis\sum\limits_{i,j=1}^\ell
v_{\va,x^i}v_{\va,y^j}\Bigr).
$$
On the other hand, by \eqref{eq42} and Proposition~\ref{pro11}, we
see
\begin{equation}\label{eq46}
\begin{array}{ll}
\dis\int_{\R^3} U_{\va,x^i}U_{\va,y^j}&=\dis\int_{\R^3}
\bigl(U_{x^i}+\va^2
p_\va(|x-x^i|)+w(|x-x^i|)\bigr)\vspace{0.2cm}\\
&\hspace{1cm}\times\bigl(U_{y^j}+\va^2
p_\va(|x-y^j|)+w(|x-y^j|)\bigr)\vspace{0.2cm}\\
&=\dis\int_{\R^3} U_{x^i}U_{y^j}+O(\va e^{-(1-\tau)|x^1-y^1|}+\va^4)\vspace{0.2cm}\\
&= \bar C_{ij}e^{-|x^i-y^j|}+O(\va e^{-(1-\tau)|x^1-y^1|}+\va^4),
\end{array}
\end{equation}
and similarly,
$$
\int_{\R^3}\dis\sum\limits_{i,j=1}^\ell
v_{\va,x^i}v_{\va,y^j}=O(\va^2 e^{-(1-2\tau)|x^1-y^1|}+\va^5).
$$
Hence,
\begin{equation}\label{46}
I_3=\dis\sum\limits_{i,j=1}^\ell \bar C_{ij}e^{-|x^i-y^j|}+O(\va e^{-(1-\tau)|x^1-y^1|}+\va^4).
\end{equation}

At last, we estimate $I(U_{\va,r}, v_{\va,r})+I(U_{\va,\rho},
v_{\va,\rho})$. We find
\begin{eqnarray*}
I(U_{\va,r}, v_{\va,r})&=&\dis\sum\limits_{j=1}^\ell
I(U_{\va,x^j}, v_{\va,x^j})-\frac14\dis\int_{\R^3}\Bigl[\Bigl(\dis\sum\limits_{j=1}^\ell
U_{\va,x^j}\Bigr)^4-\dis\sum\limits_{j=1}^\ell
U^4_{\va,x^j}-4\dis\sum\limits_{i<j}^\ell U^3_{\va,x^j}U_{\va,x^i}\Bigr]\\
&&-\frac14\dis\int_{\R^3}\Bigl[\Bigl(\dis\sum\limits_{j=1}^\ell
v_{\va,x^j}\Bigr)^4-\dis\sum\limits_{j=1}^\ell
v^4_{\va,x^j}-4\dis\sum\limits_{i<j}^\ell v^3_{\va,x^j}v_{\va,x^i}\Bigr]\\
&=&\dis\sum\limits_{j=1}^\ell I(U_{\va,x^j}, v_{\va,x^j})\\
&&-\dis\int_{\R^3}\Bigl(\dis\sum\limits_{i<j}^\ell
U^3_{\va,x^j}U_{\va,x^i}+3\dis\sum\limits_{i,j=1}^\ell
U^2_{\va,x^j}U^2_{\va,x^i}+\dis\sum\limits_{i<j}^\ell
v^3_{\va,x^j}v_{\va,x^i}+3\dis\sum\limits_{i,j=1}^\ell
v^2_{\va,x^j}v^2_{\va,x^i}\Bigr).
\end{eqnarray*}

Similar to \eqref{eq46}, we have for $i\neq j$
\begin{eqnarray*}
\dis\int_{\R^3}\dis\sum\limits_{i<j}^\ell
U^3_{\va,x^j}U_{\va,x^i}&=&\dis\sum\limits_{i<j}^\ell C_{ij}\frac{e^{-|x^i-x^j|}}{|x^i-x^j|}+O(\va
e^{-(1-\tau)|x^1-x^2|}+\va^4),\\
\dis\int_{\R^3}\dis\sum\limits_{i,j=1}^\ell
U^2_{\va,x^j}U^2_{\va,x^i}&=&\dis\sum\limits_{i<j}^\ell C'_{ij}\frac{e^{-2|x^i-x^j|}}{|x^i-x^j|^2}+O(\va
e^{-(1-\tau)|x^1-x^2|}+\va^4),
\end{eqnarray*}
and
$$
\dis\int_{\R^3}\Bigl(\dis\sum\limits_{i<j}^\ell
v^3_{\va,x^j}v_{\va,x^i}+3\dis\sum\limits_{i,j=1}^\ell
v^2_{\va,x^j}v^2_{\va,x^i}\Bigr)=O(\va^4).
$$

Therefore,
\begin{equation}\label{48}
I(U_{\va,r}, v_{\va,r})=\dis\sum\limits_{j=1}^\ell I(U_{\va,x^j},
v_{\va,x^j})+\dis\sum\limits_{i<j}^\ell C_{ij}\frac{e^{-|x^i-x^j|}}{|x^i-x^j|}+O(\va
e^{-(1-\tau)|x^1-x^2|}+\va^4).
\end{equation}
With the same argument, we check
\begin{equation}\label{49}
I(U_{\va,\rho}, v_{\va,\rho})=\dis\sum\limits_{j=1}^\ell
I(U_{\va,y^j},
v_{\va,y^j})+\dis\sum\limits_{i<j}^\ell C_{ij}\frac{e^{-|y^i-y^j|}}{|y^i-y^j|}+O(\va
e^{-(1-\tau)|x^1-x^2|}+\va^4).
\end{equation}

Now, inserting \eqref{44}, \eqref{45}, \eqref{46}, \eqref{48} and
\eqref{49} into \eqref{43}, we complete the proof.

\end{proof}

\begin{proposition}\label{491}
We have
\[
 \begin{split}
&\bar
I(U_{\va,r}+v_{\va,\rho}+\omega_\va,\omega_{\va,r}+U_{\va,\rho}+v_\va,
v_{\va,r}+\omega_{\va,\rho}+U_\va)\\
=&\dis\sum\limits_{j=1}^\ell \bar I(U_{\va,x^j}, v_{\va,x^j},
\omega_{\va,x^j})+\dis\sum\limits_{j=1}^\ell \bar I(U_{\va,y^j},
v_{\va,y^j}+\omega_{\va,y^j})+\bar I(U_\va,v_\va,\omega_\va)\\
&-\dis\sum\limits_{i<j}^\ell C_{ij}\frac{e^{-|x^i-x^j|}}{|x^i-x^j|}
-\dis\sum\limits_{i<j}^\ell C_{ij} \frac{e^{-|y^i-y^j|}}{|y^i-y^j|}+
\va\dis\sum\limits_{i,j=1}^\ell \bar C_{ij} e^{-|x^i-y^j|}
+\va\dis\sum\limits_{j=1}^\ell \tilde C_{j}(e^{-|x^j|}+e^{-|y^j|})\\
&+O\bigl(\va e^{-(1-\tau)|y^1-y^2|}+\va e^{-(1-\tau)|x^1-x^2|}+\va^2
(e^{-(1-\tau)|x^1|}+e^{-(1-\tau)|y^1|}+
e^{-(1-\tau)|x^1-y^1|})+\va^4\bigl),
\end{split}
\]
where $\bar C_{ij},\,\tilde C_{j}$ and $C_{ij} \, (i,j=1,\cdots,\ell)$ are positive constants independent
of $\va$, $r$ and $\rho$.
\end{proposition}
\begin{proof}
The proof is similar to that of Proposition~\ref{3-1} and we omit it
here.
\end{proof}
\vskip .2in

\noindent{\bf Acknowledgment}. The authors are grateful to Shusen
Yan for  helpful discussion.  S. Peng thanks Taida Institute for
Mathematical Sciences for the warm hospitality during his visit.

\end{document}